\documentclass{amsart}
\usepackage{amssymb}
\usepackage{hyperref}
\usepackage{url}
\usepackage[all]{xy}

\newtheorem{thm}{Theorem}
\newtheorem{introthm}{Theorem}
\newtheorem{introcor}[introthm]{Corollary}
\newtheorem{lem}[thm]{Lemma}
\newtheorem{cor}[thm]{Corollary}
\newtheorem{prop}[thm]{Proposition}

\theoremstyle{definition}
\newtheorem{defn}[thm]{Definition}
\newtheorem{rem}[thm]{Remark}

\numberwithin{thm}{section}
\numberwithin{equation}{section}

\newfont{\cyrr}{wncyr10}
\def\Sh{\mbox{\cyrr Sh}}

\def\Z{\mathbf{Z}}
\def\Q{\mathbf{Q}}
\def\F{\mathbf{F}}
\def\R{\mathbf{R}}

\def\bA{\mathbf{A}}

\def\Fp{\F_p}

\def\cS{\mathcal{S}}
\def\A{\mathcal{A}}

\def\O{\mathcal{O}}

\def\cP{\mathcal{P}}

\def\B{\mathcal{B}}

\def\ld{\mathcal{h}}
\def\rd{\mathcal{i}}

\def\l{\mathfrak{q}}

\def\p{\mathfrak{p}}

\def\D{\mathcal{D}}
\def\d{\mathfrak{d}}

\def\Hom{\mathrm{Hom}}
\def\Gal{\mathrm{Gal}}

\def\rk{\mathrm{rk}}

\def\Aut{\mathrm{Aut}}
\def\Sel{\mathrm{Sel}}
\def\End{\mathrm{End}}
\def\Res{\mathrm{Res}}
\def\Frob{\mathrm{Frob}}

\def\ur{\mathrm{ur}}
\def\GL{\mathrm{GL}}
\def\SL{\mathrm{SL}}

\def\ST{\Sh\mathrm{T}}
\def\loc{\mathrm{loc}}

\def\image{\mathrm{image}}
\def\Pic{\mathrm{Pic}}
\def\ram{\mathrm{ram}}

\def\even{\mathrm{even}}

\def\parity{\rho}
\def\rP{\mathbf{E}}

\def\Ceb{\mathcal{L}}

\def\rE#1{#1^+}
\def\rO#1{#1^-}
\def\Podd{\rO{\rP}}
\def\Peven{\rE{\rP}}

\def\N{\mathbf{N}}

\def\Kb{\bar{K}}

\def\too{\longrightarrow}
\def\map#1{\;\xrightarrow{#1}\;}
\def\isom{\xrightarrow{\sim}}
\def\hookto{\hookrightarrow}
\def\onto{\twoheadrightarrow}
\def\dirsum#1{\underset{#1}{\textstyle\bigoplus}}

\def\Hu{H^1_{\ur}}

\def\HF{H^1_{\cS}}

\def\bmu{\boldsymbol{\mu}}
\def\bq{\mathbf{q}}
\def\Fset{\mathcal{F}}
\def\Xset{\mathcal{C}}

\def\H{\mathcal{H}}

\def\sgn{\mathrm{sign}_\Delta}
\def\sgnd{\mathrm{sign}_{\Delta}}

\def\one{\mathbf{1}}
\def\iK{\bA_K^\times}

\def\metabolic{metabolic }

\def\dm#1{\| #1 \|}
\def\ML{M_L}
\def\PS{W}
\def\disparity{\delta}

\title[A Markov model for Selmer ranks in families of twists]{A Markov model 
for Selmer ranks \\ in families of twists}

\author{Zev Klagsbrun}
\address{Department of Mathematics, 
University of Wisconsin\,-\,Madison,
Madison, WI 53706, 
USA}
\email{\href{mailto:klagsbru@math.wisc.edu}{klagsbru@math.wisc.edu}}
\author{Barry Mazur}
\address{Department of Mathematics, 
Harvard University,
Cambridge, MA 02138, 
USA}
\email{\href{mailto:mazur@math.harvard.edu}{mazur@math.harvard.edu}}
\author{Karl Rubin}
\address{Department of Mathematics, 
UC Irvine,
Irvine, CA 92697, 
USA}
\email{\href{mailto:krubin@math.uci.edu}{krubin@math.uci.edu}}


\thanks{This material is based upon work supported by the 
National Science Foundation under grants DMS-0700580, DMS-0757807, DMS-0968831, and DMS-1065904.  
Much of this work was carried out while the second and third authors were in residence at 
MSRI, and they would also like to thank MSRI for support and hospitality.}

\begin{document}

\begin{abstract}
We study the distribution of $2$-Selmer ranks in the family of quadratic twists of 
an elliptic curve $E$ over an arbitrary number field $K$. 
Under the assumption that $\Gal(K(E[2])/K) \cong S_3$ we show that the density 
(counted in a non-standard way) of twists with Selmer rank $r$ exists for all 
positive integers $r$, and is given via an equilibrium distribution, depending only on 
a single parameter (the ``disparity''), of a certain Markov process that is itself 
independent of $E$ and $K$.
More generally, our results also apply to $p$-Selmer ranks of twists of $2$-dimensional 
self-dual ${\bf F}_p$-representations of the absolute Galois group of $K$ by characters 
of order $p$.
\end{abstract}

\maketitle

\setcounter{tocdepth}{1}
\tableofcontents

\section*{Introduction}

There has been much recent interest in the arithmetic statistics related to the 
class of all elliptic curves over a given number field. For example, there are 
the spectacular results due to Bhargava and Shankar \cite{bs1,bs2} over ${\Q}$. 
There are also precise and extensive statistical conjectures (cf. \cite{poonenrains,bklpr}) 
proposing that density distributions of ranks of $p$-Selmer groups 
are given by equilibrium distributions arising from certain Markov processes.

This article deals with the statistical shape of the ranks of $2$-Selmer groups 
in the family of quadratic twists of a given elliptic curve $E$ over a given number field 
$K$ (that is, twists of $E$ by all quadratic characters of $K$). 

Define the {\em disparity} $\disparity(E/K)$ 
of such a family to be the difference between $1/2$ and 
the density of the members with even $2$-Selmer rank. We showed in \cite[Theorem 7.6]{kmr} that 
when one orders the members of such a quadratic twist family in a certain natural 
way, this disparity---i.e., such a ``density"---exists, and we gave an example 
of a curve $E$ such that, as $K$ varies, the disparity 
takes on a dense set of values in its allowable range 
$[-\frac{1}{2}, {\frac{1}{2}}]$. 
(On the other hand, when $K = \Q$ the disparity is always zero.)  
Conjecturally, then, this would also imply the same 
facts for Mordell-Weil ranks of the members of these families. 

\subsection*{Our main result}
This paper is a sequel to \cite{kmr}. We prove:

\begin{introthm}
\label{thmb}
Let $E$ be an elliptic curve over a number field $K$ with
$$
\Gal(K(E[2])/K) \cong S_3.
$$ 
For every $m \ge 0$ and $X > 0$ 
let $m \mapsto\B_m(X) = \cup_k \B_{m,k,X}$ be the ``fan-structure'' of 
collections of quadratic characters of $K$ as in Corollary \ref{11.11}.  
Then for every $r \ge 0$,
\begin{multline*}
\lim_{m \to \infty}\lim_{X \to \infty}\frac{|\{\chi\in\B_m(X) : 
   \dim_{\F_2}\Sel_2(E^\chi/K) = r\}|}{|\B_m(X)|} \\
   = \begin{cases}(\frac{1}{2} + \disparity(E/K))c_r&\text{if $r$ is odd},\\
   (\frac{1}{2}-\disparity(E/K))c_r&\text{if $r$ is even,}\end{cases}
\end{multline*}
where $c_r$ is the positive real number given by Definition \ref{10.3} with $p = 2$.
\end{introthm}

In other words, the only parameter needed to fully describe the distribution of $2$-Selmer ranks in 
the family of quadratic twists of $E$ (at least in the case when $\Gal(K(E[2])/K) \cong S_3$)
is the disparity $\disparity(E/K)$.  
A similar result, with the same constants $c_r$ (but where the disparity 
$\disparity(E/K)$ is necessarily $0$)  was obtained by 
Swinnerton-Dyer \cite{sw-d} in the case where the number field was 
${\Q}$ and the Galois action on $2$-torsion was trivial.

\subsection*{Fan structure}
In section \ref{avgs} below we define the set of {\it levels} $\D$ 
(eventually associated to quadratic characters) for the field $K$ and we axiomatize an assignment of subsets 
$$
(m,k,X) \mapsto  \D_{m,k,X} \;\subset\;\D
$$ 
for triples $(m,k,X)$ (for integers $m, k \ge 0$ and positive real values $X$) called a 
{\em fan structure} on $\D$. We consider subsets, $\B_{m,k,X}$,  of the group 
of quadratic characters over $K$ related---according to a certain cuisine---to 
the $\D_{m,k,X}$. We study average $2$-Selmer ranks of twists of  $E$, where 
we twist by collections of quadratic characters of the form $\B_m(X) = \cup_k\B_{m,k,X}$.  
See \S\ref{rankdensities}, especially Definition \ref{11.7} and Corollary \ref{11.11}, below.  
The reason for the adjective `fan' is that the subscript 
$m$ refers to the number of ramified prime divisors in the twisting characters and 
as $m$ increases, our method requires us to average over characters divisible by 
primes of larger and larger norms. The successive primes are allowed to 
`fan out'---so to speak---being subject to increasing upper bounds for the 
absolute value of their norms, this increase  being dictated inductively by 
effective Cebotarev estimates. 

\subsection*{On the ordering of twists}

Perhaps the most natural order of {\it all} elliptic curves over 
a given number field is via the size of the absolute value of the 
conductor of the elliptic curve. In the special context of 
Swinnerton-Dyer's theorem \cite{sw-d} it is a result of Kane \cite{kane}  
(see also \cite{heath-brown}) that one obtains the same arithmetic 
statistics if one orders twists in this manner, rather than ordering 
them the way Swinnerton-Dyer does. Specifically the disparity 
(which remains $0$ in this context) and the $c_r$'s are the 
same as in Swinnerton-Dyer's original theorem.

Something different happens in our more general context. If 
one orders quadratic twists by the 
norm of their conductor, rather than by the largest norm of  
any prime dividing the conductor, 
the disparity may very well change (see \cite[Example 7.13]{kmr}). 
It is conceivable, however, that the relative $2$-Selmer rank 
densities  still exist and are  as dictated by the (appropriately changed) 
disparity and the same numbers $c_r$ as above.

\subsection*{Average Mordell-Weil rank} 
Since the $2$-Selmer rank is an upper bound for the Mordell-Weil rank, 
Theorem \ref{thmb} has the following immediate corollary.

\begin{introcor}
Suppose that $E$ is an elliptic curve over a number field $K$, and that 
$\Gal(K(E[2])/K) \cong S_3$.  With notation as in Theorem \ref{thmb}, 
the average rank of the twists of $E$ satisfies
$$
\lim_{m \to \infty}\lim_{X \to \infty} 
   \frac{\sum_{\chi\in\B_m(X)} \rk(E^\chi(K))}{|\B_m(X)|}
   \;<\; 1.2646 + 0.1211 \cdot \disparity(E/K) \;<\; 1.3252. \\
$$
\end{introcor}

\subsection*{How generally are these densities Markovian?}
A future project is to understand the extent to which Markov models  
suffice to explain phenomena in contexts of greater generality than we treat here.
  
For example, considering the four different possible types of images of the 
Galois group in $\Aut(E[2]) \cong S_3$, one expects that each case has its 
interesting story. For the case when the image is of order $2$ see 
forthcoming work of the first author \cite{klagsbrun}.
   
One would also want to see this project extended to deal with abelian varieties 
of general dimension. A few lucky accidents, however, happen in dimension one 
that allow us to prove our theorem. 
To explain these accidents we  briefly sketch our method. 

The $2$-Selmer 
group of an elliptic curve $E$ over a number field $K$ is given 
by imposing {``local conditions''} at every place $v$ of $K$, and 
restricting to the subgroup of $H^1(G_K,E[2])$ consisting of cohomology classes 
that satisfy those local conditions at all places.  Twisting $E$ by a 
quadratic character $\chi$ of $K$ does not change the ${\F}_2[G_K]$-module 
$E[2]$, but can (and usually does) change some of the local conditions. It is 
natural, when studying statistics of the ${\F}_2$-dimensions of the Selmer groups 
of these twisted elliptic curves $E^{\chi}$, to first consider the statistics 
of a larger collection of objects, namely of the subspaces of $H^1(G_K,E[2])$ 
subject to what we call an arbitrary {\it Selmer structure}; namely, where for 
a given finite set of places $S$ containing all places dividing $2\infty$ 
and all places of bad reduction for $E$ we impose what one might call 
{``incoherent''} local conditions on the cohomology groups $H^1(G_{K_v},E[2])$ 
by twisting by local quadratic characters $\chi_v$ for $v \in S$, 
retaining the natural local condition at all other places. 
Such a collection of local quadratic characters $\{\chi_v\}_{v\in S}$ may or 
may not be {``coherent''} in the sense that the package $\{\chi_v\}_{v\in S}$ comes 
(by restriction) from a single global quadratic character unramified outside $S$.  
Our method consists in understanding how ranks of these incoherent 
$2$-Selmer groups change as we twist by 
one local character $\chi_v$ at a time. Our Markov process is 
precisely this successive twisting.
     
The way we pass from statistics regarding this large class of incoherent 
Selmer structures to the ones that have global meaning uses 
what we might call {``free''} places $v$. 
A free place $v$ is one where twisting by $\chi_v$ doesn't change 
the local Selmer condition, and hence doesn't change the $2$-Selmer rank. 
The assumption that $E(K)$ has no points of 
order $2$ guarantees that there are enough free places so that 
every incoherent package of local quadratic characters can be augmented by an 
appropriate assortment of characters at free places to render the augmented collection coherent, 
without changing the $2$-Selmer rank. Roughly speaking, averaging over the free places 
allows us to convert rank statistics for incoherent $2$-Selmer groups to 
rank statistics for $2$-Selmer groups of quadratic twists of elliptic curves.

Suppose now that $A$ is a principally polarized abelian variety of dimension $g$, 
and $v \nmid 2\infty$ is a prime of good reduction.  Then  
the local cohomology group $H^1(G_{K_v},A[2])$ is a quadratic space of dimension $2d$, 
where $0 \le d \le 2g$.  The local Selmer condition for the twist of $A$ by $\chi_v$ 
is a Lagrangian subspace of $H^1(G_{K_v},A[2])$.  
There is a canonical Lagrangian subspace $V_{\ur}$, 
the unramified space, which is the local condition if $\chi_v$ is unramified.  
If $\chi_v$ is ramified, then the local condition is a Lagrangian subspace 
whose intersection with $V_{\ur}$ is zero.  A calculation of Poonen and Rains 
\cite[Proposition 2.6]{poonenrains} shows that there are $2^{d(d-1)/2}$ 
such spaces.

When $d=0$, all the local conditions are necessarily zero, so the $2$-Selmer group
is independent of $\chi_v$; these are exactly the free 
places discussed above.  When $d = 1$, there is only one possibility for the 
local condition when $\chi_v$ is ramified.  When $d = 2$, there are two possibilities, 
and one can show that these correspond to the $2$ ramified characters $\chi_v$.  
We don't know which ramified character corresponds to which Lagrangian, but since  
we are averaging over all the local characters, we don't need to.  
If $A$ is elliptic curve, then $d \le 2$, so this covers all cases.

However, if $g > 1$, then $d$ can be greater than $2$.  
In that case there are more than $2$ possible 
ramified Lagrangians, but only $2$ ramified local characters.  
Thus without additional information in this higher-dimensional case, 
we don't know how to average the Selmer rank over the local characters.

\subsection*{How generally are densities determined by Cebotarev conditions?}
It seems likely that the finer question 
of how the Selmer rank changes under twist by a single ramified character
is not determined by Cebotarev conditions alone!  See \cite[\S10]{spin}.

\subsection*{Is an elliptic curve determined (up to isogeny) by the 
Selmer ranks of its twists?}  Theorem \ref{thmb} shows that 
the distribution of $2$-Selmer ranks 
is independent of the elliptic curve $E$ over $\Q$, and over a general number field 
depends only on a single parameter, the disparity.  This leads one to ask 
how much the actual function $\chi \mapsto \dim_{\F_2}\Sel_2(E^\chi)$ 
determines about $E$.  For example, how often do the rank functions of 
two non-isogenous elliptic curves coincide?  The answer seems to be: sometimes, 
but not often.  For a discussion of this question,  
some sufficient conditions for non-isogenous 
elliptic curves to share the same rank function, and some examples, 
see \cite{companions}.

\subsection*{The layout of the paper}  Although our main interest is 
$2$-Selmer ranks of quadratic twists of elliptic curves, our methods also 
apply to more general Selmer groups attached to $2$-dimensional self-dual 
$\F_p[G_K]$-modules, so we work in this generality.

The first part of the paper is purely combinatorial.
In \S\ref{pd} we introduce some notation and very basic facts about probability distributions 
and Markov processes, and in \S\ref{ourM} we introduce the particular 
Markov process that will govern our Selmer rank statistics.  In \S\ref{avgs} 
we axiomatize the kind of counting structure that will arise for our families of twists, 
and in \S\ref{fanaverage} we prove our basic results (Theorem \ref{ybox} and 
Corollary \ref{4.6}) about averages in 
this general setting.

The second part of the paper contains all the arithmetic.  Section \ref{notation} 
describes the general setup of the Selmer groups we will consider, and \S\ref{eex} 
shows how twists of elliptic curves fit into this setup.  In \S\ref{csr} we 
describe how the Selmer rank changes when we change a single local condition, 
and in \S\ref{lgc} we use class field theory to show that the average over all 
local twists (incoherent Selmer structures, in the description above) 
is the same as the average over twists by global characters.  
Finally in \S\ref{rankdensities} we tie everything 
together to prove Theorem \ref{thmb} and related results.

\part{Markov processes and fan structures}
\label{part1}

\section{Probability distributions}
\label{pd}

\begin{defn}
View $\Z_{\ge 0}  = \{ 0,1,2,\dots\}$  as a $\sigma$-finite measure space, with each 
point $ x\in \Z_{\ge 0}$ having measure $1$.  Form the Banach space over $\R$
$$
\ell^1: = L^1(\Z_{\ge 0} ) = 
\{\text{set maps $f: \Z_{\ge 0} \to \R$ such that $\dm{f}:= \sum_{n\ge 0}|f(n)|$ converges}\}.
$$
Let $\PS \subset \ell^1$ denote the closed convex subspace of {\em densities}, 
or {\em probability distributions},
$$
\PS : = \{ f \in \ell^1 : \text{$f(n) \ge 0$ for all $n\in \Z_{\ge 0}$ and $\|f\|=1$}\}.
$$
A bounded linear operator $M : \ell^1 \to \ell^1$ is called a {\em Markov operator} if 
$M(\PS) \subset \PS$.
We can write $M$ as an infinite matrix $[m_{r,s}]_{r,s\in\Z_{\ge 0}}$ 
where, for $f \in \ell^1$, 
$$
(M(f))(s) = \sum_{r \ge 0} m_{r,s}f(r),
$$  
with $\{m_{r,s}\}$ bounded, and then $M$ is a  Markov operator if and only if 
$m_{r,s} \ge 0$ for all $r,s \ge 0$ and $\sum_{s \ge 0}m_{r,s} = 1$ for every $r$.  
\end{defn}

\begin{defn}
If $f \in \PS$, we define the {\em parity} $\parity(f)$ of $f$ by
$$
\parity(f) := \sum_{\text{$n$ odd}}f(n).
$$
Let $\PS^+, \PS^- \subset \PS$ be the subsets
$$
\arraycolsep=2pt
\renewcommand{\arraystretch}{1.25}
\begin{array}{rcl}
\PS^+ := & \{f \in \PS : \text{$f(n) = 0$ if $n$ is odd}\} &= \{f \in \PS : \parity(f) = 0\}, \\
\PS^- := &\{f \in \PS : \text{$f(n) = 0$ if $n$ is even}\} &= \{f \in \PS : \parity(f) = 1\}. \\
\end{array}
$$
We say that a Markov operator $M$ is {\em parity preserving} if $m_{r,s} = 0$ whenever $r \not\equiv s \pmod{2}$, 
and $M$ is {\em parity reversing} if $m_{r,s} = 0$ whenever $r \equiv s \pmod{2}$.

Define operators $\pi^+$, $\pi^-$ on $\ell^1$, $\pi^+ + \pi^- = 1$, by 
$$
\pi^{+}_{r,s} = 
\begin{cases}
1 & \text{if $i=j$ and $i$ is even,} \\
0 & \text{otherwise},
\end{cases}
\quad
\pi^{-}_{r,s} = 
\begin{cases}
1 & \text{if $i=j$ and $i$ is odd,} \\
0 & \text{otherwise}.
\end{cases}
$$
\end{defn}

\begin{lem}
\label{wehave}
Suppose $M$ is a Markov operator and $f \in \PS$.  
\begin{enumerate}
\item
If $M$ is parity preserving, then $M(\PS^\pm) \subset \PS^\pm$, 
$\parity(M(f)) = \parity(f)$, and $M \circ\pi^\pm = \pi^\pm \circ M$, 
\item
if $M$ is parity reversing, then $M(\PS^\pm) \subset \PS^\mp$, 
$\parity(M(f)) = 1-\parity(f)$, and $M \circ\pi^\pm = \pi^\mp \circ M$, 
\item
$\pi^+(f) \in (1-\parity(f))\PS^+$ and
$\pi^-(f) \in \parity(f)\PS^-$.
\end{enumerate}
\end{lem}

\begin{proof}
Exercise.
\end{proof}

\section{Example: the mod $p$ Lagrangian operator $M_L$}
\label{ourM}

Fix a prime $p$.

\begin{defn}
\label{MLdef}
Define a bounded operator $\ML = [m_{r,s}]$ on $\ell^1$ by 
$$
m_{r,s} = 
\begin{cases}
1-p^{-r} & \text{if $s = r-1 \ge 0$}, \\
p^{-r} & \text{if $s = r+1 \ge 1$}, \\
0 & \text{otherwise}.
\end{cases}
$$
\end{defn}

Then $\ML$ is a parity reversing Markov operator, and 
$\ML^2$ is a parity preserving Markov operator.  
We call $\ML$ the mod $p$ Lagrangian operator.

\begin{defn}
\label{10.3}
For $n \ge 0$ define
$$
c_n := \prod_{j=1}^\infty (1+p^{-j})^{-1} \prod_{j=1}^n \frac{p}{p^j-1}.
$$
Define $\Peven, \Podd \in \ell^1$ by
$$
\Peven(n) := \begin{cases} c_n & \text{if $n$ is even} \\ 0 & \text{if $n$ is odd},\end{cases} \quad
\Podd(n) := \begin{cases} 0 & \text{if $n$ is even} \\ c_n & \text{if $n$ is odd}.\end{cases}
$$
\end{defn}

\begin{lem}
\label{oddeven}
\begin{enumerate}
\item
$\Peven \in \PS^+$ and $\Podd \in \PS^-$.
\item
$\ML(\Peven) = \Podd$ and $\ML(\Podd) = \Peven$.
\item
$\ML^2(W^+) \subset W^+$ and $\ML^2(W^-) \subset W^-$.
\end{enumerate}
\end{lem}

\begin{proof}
For (i), we only need to show that $\sum_n \Peven(n) = \sum_n \Podd(n) = 1$.  
See \cite[Proposition 2.6]{poonenrains}, or \cite{heath-brown} for the case $p=2$.

It follows directly from the definitions that $\ML(\Peven)(n) = 0$ 
if $n$ is even.  If $n$ is odd, then using that $c_{n+1}/c_n = p/(p^{n+1}-1)$ we have
\begin{multline*}
\ML(\Peven)(n) = c_n \biggl((1-p^{-1-n})\frac{p}{p^{n+1}-1} + p^{1-n}\frac{p^n-1}{p}\biggr) \\
   = c_n(p^{-n} + (1-p^{-n})) = c_n.
\end{multline*}
Thus $\ML(\Peven) = \Podd$, and in exactly the same way $\ML(\Podd) = \Peven$.

The third assertion is clear.
\end{proof}

\begin{prop}
\label{10.5}
For every $f \in \PS$, 
\begin{align*}
\lim_{k \to \infty} \ML^{2k}(f) &= (1-\parity(f)) \Peven + \parity(f) \Podd, \\
\lim_{k \to \infty} \ML^{2k+1}(f) &= \parity(f) \Peven + (1-\parity(f)) \Podd.
\end{align*}
In particular if $\parity(f) = \frac{1}{2}$, then 
$\lim_{k \to \infty} \ML^k(f) = \frac{1}{2}\Podd + \frac{1}{2}\Peven$.
\end{prop}

\begin{proof}
By Lemma \ref{oddeven}(iii), we can view $\ML^2$ as a Markov process on $\Z_{\ge 0}^\even$, 
and by Lemma \ref{oddeven}(i), $\Peven \in \PS^+$ is an equilibrium state for this 
Markov process (i.e., $\ML^2(\Peven) = \Peven$).  
This Markov process is irreducible and aperiodic on $\Z_{\ge 0}^\even$ in the sense of 
\cite[Chapter 1]{norris}.  By \cite[Theorem 1.8.3]{norris}, it follows that the 
equilibrium distribution is unique, and that for every $f \in \PS^+$ we have
$$
\lim_{k \to \infty} \ML^{2k}(f) = \Peven.
$$ 
In exactly the same way, $\Podd \in \PS^-$ is the unique equilibrium state for $\ML^2$ in $\PS^-$ 
and for every $f \in \PS^-$ we have
$
\lim_{k \to \infty} \ML^{2k}(f) = \Podd.
$
Now the proposition follows from Lemma \ref{wehave}(ii,iii).
\end{proof}

\begin{rem}
Our description of Markov processes is limited to Markov 
operators that act on the set of probability distributions. 
One can more generally define Markov operators as infinite 
matrices satisfying the conditions appearing immediately 
prior to Definition 1.2, that act on arbitrary sequences of 
non-negative real numbers.

Some of the techniques we develop here can also be applied 
to such Markov operators, assuming that the operator under 
consideration has a unique (up to scalar multiple) 
equilibrium state. See the 
forthcoming work of the first author and Valko \cite{KlagsbrunValko} 
for an arithmetic application of such a case.
\end{rem}

\section{Axiomatizing the Markovian counting setup}
\label{avgs}

In this section we axiomatize the kind of general argument that we will use to 
find the distribution of Selmer ranks corresponding to (``incoherent'', as discussed in 
the Introduction) twists of an elliptic curve.  

Fix an elliptic curve $A$ defined over a number field $K$, and a rational prime $p$. 
To motivate the definitions below, we illustrate each one by 
giving its interpretation in the elliptic curve case, i.e, 
the case of Selmer ranks attached to twists of $A[p]$.

\subsection{Normed set with linear growth}

\begin{defn}
\label{nsg}
A {\em normed set} is a set $S$ together with a real-valued norm function $\N : S \to \R_{> 0}$.  
If $S$ is a normed set, we define $S(X) := \{s \in S : \N(s) < X\}$, and we say that $S$ has 
{\em linear growth} if for every $\epsilon > 0$, 
\begin{equation}
\label{growth}
X^{1-\epsilon} < |S(X)| < X^{1+\epsilon} \qquad \text{for $X \gg_\epsilon 1$}.
\end{equation}
\end{defn}

The norm provides the fundamental ordering that will allow us to take averages. 

Fix a normed set $\cP$ with linear growth.

\begin{rem}
In the elliptic curve case, let $\Sigma$ be a finite set of places of $K$ including all 
nonarchimedean places, all primes where $A$ has bad reduction, and all primes above $p$. 
Then $\cP$ will be the set of all primes of $K$ not in $\Sigma$, with the usual 
(absolute) norm function.  These primes correspond to ``minimal'' twists.
\end{rem}

\subsection{Width}

\begin{defn}
By a {\em width function} $w : \cP \to \Z_{\ge 0}$ we mean a function with finite image $I$, 
and such that for each $i \in I$, the inverse image $\cP_i := w^{-1}(i)$ with the induced 
norm function $\N$ is a normed set with linear growth.  
\end{defn}

Fix a width function $w$ on $\cP$.

\begin{rem}
In the elliptic curve case, if $\l$ is a prime in $\cP$ we define
$$
w(\l) := 
\begin{cases}
0 & \text{if $\bmu_p \notin K_\l^\times$}, \\
\dim_{\Fp}A(K_\l)[p] & \text{if $\bmu_p \in K_\l^\times$}.  
\end{cases}
$$
Then $\{2\} \subset I \subset \{0,1,2\}$, and if $i \in I$ 
then $\cP_i$ has linear growth by the Cebotarev theorem.
The width $w(\l)$ is the largest possible change in Selmer rank when 
we twist by a local character at $\l$.
\end{rem}

\subsection{Levels}
\label{1.C}

\begin{defn}
A finite subset of $\cup_{i > 0}\cP_i =
 \{q\in\cP : w(q) > 0\}$ will be called a {\em level}.  Denote by $\D$ the 
{\em set of levels}, i.e., the set of all finite subsets of $\cup_{i > 0}\cP_i$.
We extend $w$ and $\N$ from $\cP$ to $\D$ by 
$w(\delta) = \sum_{q\in\delta}w(q)$ and $\N(\delta) = \prod_{q\in\delta}\N(q)$.
\end{defn} 

\begin{rem}
In the elliptic curve case, the levels correspond to square-free ideals supported on 
$\cP_1 \cup \cP_2$.  
If $\chi$ is a quadratic character of $K$, then the level of $\chi$ is the part of the 
conductor of $\chi$ supported on $\cP_1 \cup \cP_2$.  

We exclude primes of width zero from the level because twisting by a 
prime of width zero has no effect on the Selmer group, either because 
all such characters are unramified (if $\bmu_p \notin K_\l^\times$) 
or because $H^1(K_\l,A[p]) = 0$ (if $A(K_v)[p] = 0$).
\end{rem}

\subsection{Rank data}

\begin{defn}
\label{rsdef}
By {\em rank data} on $\D$ we mean a rule that assigns to every level $\delta\in\D$ a finite 
set $\Omega_\delta$, together with the following extra structure:
\begin{itemize}
\item
a map (called the {\em rank map}) $\rk : \Omega_\delta \to \Z_{\ge 0}$ for every $\delta$,
\item
a map $\eta_{\delta,q} : \Omega_{\delta \cup \{q\}} \to \Omega_\delta$ for every $\delta\in\D$ 
and $q\in \cP-\delta$, such that all fibers $\eta_{\delta,q}^{-1}(\omega)$ 
have cardinality independent of $\delta$, $q$ and $\omega$.
\end{itemize}
\end{defn}

Note that it follows from the second property of Definition \ref{rsdef} that 
if $|\delta| = |\delta'|$ then $|\Omega_\delta| = |\Omega_{\delta'}|$.

Fix rank data on $\D$.

\begin{rem}
In the elliptic curve case, for $\delta \in \D$ we set
$$
\Omega_\delta = \{\omega = (\omega_v) \in \prod_{v \in \Sigma \cup \delta} \Hom(K_v^\times,\bmu_p) 
   : \text{$\omega_\l$ is ramified if $\l \in \delta$}\}
$$
(we say that $\omega_\l$ is ramified if it is nontrivial on $\O_\l^\times$, 
the local units in $K_\l^\times$).
The rank map is given by $\rk(\omega) := \dim_{\Fp}\Sel(A[p],\omega)$, 
where $\Sel(A[p],\omega)$ is the twisted Selmer group given by Definition \ref{sstwist} below,  
and the 
map $\eta_{\delta,\l} : \Omega_{\delta \cup \{\l\}} \to \Omega_\delta$ is the forgetful map
that simply drops $\omega_\l$.  Since $w(\l) > 0$, there are exactly $p^2-p$ ramified
characters of $K_\l^\times$, so all fibers $\eta_{\delta,\l}^{-1}(\omega)$ have size $p^2-p$.
\end{rem}

\subsection{Rank distribution function}

\begin{defn}
Given rank data on $\D$, the corresponding 
{\em rank distribution function} is the function $E : \D \to \PS$ defined by
$$
E_\delta(r) = \frac{|\{\omega\in\Omega_\delta : \rk(\omega) = r\}|}{|\Omega_\delta|}
$$
for every $r \ge 0$.
If $B$ is a nonempty finite subset of $\D$, the 
{\em rank distribution over $B$} is the average of the $E_\delta$ over $\delta \in B$, 
weighted according to the size of $\Omega_\delta$:
$$
E_B := \frac{\sum_{\delta\in B}|\Omega_\delta|E_\delta}{\sum_{\delta\in B}|\Omega_\delta|} \in \PS.
$$
\end{defn}

Thus $E_B(r)$ is the probability, as $\delta$ ranges through $B$, that $\rk(\delta) = r$.  
If all $\delta \in B$ have the same cardinality, then all $\Omega_\delta$ have the 
same cardinality, so 
$
E_B = \frac{\sum_{\delta\in B}E_\delta}{|B|}.
$

\subsection{Governing Markov operators}

\begin{defn}
\label{govdef}
Suppose $M$ is a Markov operator. 
We say that $M$ {\em governs} the rank data $\Omega$ if for every $\delta \in \D$, every 
$\omega\in\Omega_\delta$, every $i \in I$, and every $s \in \Z_{\ge 0}$,
\begin{equation}
\label{conv}
\lim_{X \to \infty} 
   \frac{\sum_{q\in\cP_i(X)-\delta}|\{\chi\in\eta_{\delta,q}^{-1}(\omega) : \rk(\chi) = s\}|}
      {\sum_{q\in\cP_i(X)-\delta}|\eta_{\delta,q}^{-1}(\omega)|}
   = m_{\rk(\omega),s}^{(i)}
\end{equation} 
where $M^i = [m^{(i)}_{r,s}]$.
\end{defn}

To say that $M$ governs the rank data means essentially that adding a random 
$q$ affects the rank statistics in the same way as applying the operator $M^{w(q)}$. 
 
Fix a Markov operator $M$ that governs the rank data $\Omega$.  

\begin{rem}
In the elliptic curve case, under suitable hypotheses (see \eqref{bighyp0}, 
\eqref{bighyp1}, and \eqref{bighyp2} below) we will show that 
the rank data described above are governed by the mod $p$ Lagrangian Markov operator 
of Definition \ref{MLdef}.
\end{rem}

\subsection{Convergence rates}

\begin{defn}
\label{cebdef}
A {\em convergence rate} for $(\Omega,M)$ is a nondecreasing 
function $\Ceb$ from the infinite real interval $[1,\infty)$ to itself such that
for every real number $Y \ge 1$, every $\delta\in\D$ with $\N(\delta) < Y$, 
every $\omega\in\Omega_\delta$, every $i \in I$, every $s \in \Z_{\ge 0}$, 
and every $X \ge \Ceb(Y)$,
\begin{equation}
\label{cdefi}
\left|\frac{\sum_{q\in\cP_i(X)-\delta}|\{\chi\in\eta_{\delta,q}^{-1}(\omega) : \rk(\chi) = s\}|}
      {\sum_{q\in\cP_i(X)-\delta}|\eta_{\delta,q}^{-1}(\omega)|}
   - m_{\rk(\omega),s}^{(i)}\right| \le \frac{1}{Y}.
\end{equation}
\end{defn}

In other words, $\Ceb$ makes effective the rate of convergence in \eqref{conv}. 

Fix a convergence rate $\Ceb$ for $(\Omega,\ML)$.

\begin{rem}
In the elliptic curve case, we will show (see Theorem \ref{gov} below) 
that $\ML$ governs the rank data  
with a convergence rate that comes from an effective version of the Cebotarev theorem.
\end{rem}

\subsection{Stratification of levels}

\begin{defn}
\label{boxdef}
Define a sequence of real valued functions $\{L_n(Y)\}_{n \ge 1}$ by
\begin{align*}
L_1(Y) &:= \Ceb(Y), 
   \\
L_{n+1}(Y) &:= \max\{\Ceb(\textstyle\prod_{j \le n}L_j(Y)),Y L_n(Y)\}, 
   \quad n \ge 1. 
\end{align*}
If $m, k \in \Z_{\ge 0}$ and $X \in \R_{>0}$, define the ``fan''
$$
\D_{m,k,X} := \{\delta\in\D : \text{$w(\delta) = k$ and $\delta = \{q_1,\ldots,q_m\}$ 
   with $\N(q_j) < L_j(X)$ for all $j$}\}.
$$
\end{defn}

Although we suppress it from the notation, $\D_{m,k,X}$ depends on the 
(fixed) convergence rate $\Ceb$.

\section{Averages over fan structures}
\label{fanaverage}

Keep the notation of the previous section, along with the fixed prime $p$, 
normed set $\cP$, width function $w$ with image $I$, rank data $\Omega$, Markov operator 
$M$ governing $\Omega$, and convergence rate $\Ceb$ for $(\Omega,M)$.
In this section we will show how to use all of this information to compute the 
rank statistics as we average over our ``fan structures'' $\D_{m,k,X}$.

If $B \subset \D$ and $C \subset \cP$, define 
$$
B*C := \{\delta \cup \{q\} : \delta \in B, q \in C - \delta\}.
$$

\begin{rem}
For our application we would like to compute
$$
\lim_{X \to \infty} E_{\D(X)},
$$
where $\D(X) = \{\delta \in \D : \prod_{q\in\delta}\N(q) < X\}$.
Unfortunately we have not yet been able to do this.  Instead, for every 
level $\delta\in\D$ and $i\in I$ we will show (Proposition \ref{-1}) that
\begin{equation}
\lim_{X \to \infty} E_{\{\delta\} * \cP_i(X)} = M^i(E_{\{\delta\}})
\end{equation}
Using this, we will show (Theorem \ref{ybox}) that for every $m$ and $k$,
$$
\lim_{X \to \infty} E_{\D_{m,k,X}} = M^k(E_{\delta_0})
$$
where $\delta_0 = \emptyset \in \D$.
If $M = \ML$, then taking the limit as $m$ and $k$ go to infinity we can use 
Proposition \ref{10.5} to describe the limiting statistics in terms of the 
equilibrium states of $\ML$ (Corollary \ref{4.6}).
\end{rem}

\begin{prop}
\label{-1}
Suppose that 
$$
b := \sup\{\rk(\omega) : \omega\in\Omega_{\delta\cup\{q\}},q\in\cP_i\} < \infty.
$$
Then for every $Y \ge 1$, every $\delta\in\D$ with $\N(\delta) < Y$, every $i \in I$, 
and every $X \ge \Ceb(Y)$, we have the following upper bound on the $\ell^1$ norm
$$
\left\| E_{\{\delta\} * \cP_i(X)} - M^i(E_\delta)\right\| \le \frac{b+1}{Y}.
$$
\end{prop}

\begin{proof}
Fix $s \ge 0$, and let $d$ be the common value $|\eta_{\delta,q}^{-1}(\omega)|$ (independent of 
$\omega \in \Omega_\delta$ and $q \in \cP_i$).  Then
\begin{align*}
E_{\{\delta\} * \cP_i(X)}(s) &= \frac{1}{|\cP_i(X)-\delta|}\sum_{q\in\cP_i(X)-\delta} E_{\delta\cup\{q\}}(s) \\
   &= \frac{1}{|\cP_i(X)-\delta|}\sum_{q\in\cP_i(X)-\delta} \frac{|\{\omega\in\Omega_{\delta\cup\{q\}} : \rk(\omega) = s\}|}
      {|\Omega_{\delta\cup\{q\}}|} \\
   &= \frac{1}{|\cP_i(X)-\delta|}\sum_{q\in\cP_i(X)-\delta}
      \frac{\sum_{\omega\in\Omega_\delta}|\{\chi\in\eta_{\delta,q}^{-1}(\omega) : \rk(\chi) = s\}|}
         {d\;|\Omega_\delta|} \\
   &= \frac{1}{|\Omega_\delta|}\sum_{\omega\in\Omega_\delta}
      \frac{\sum_{q\in\cP_i(X)-\delta}|\{\chi\in\eta_{\delta,q}^{-1}(\omega) : \rk(\chi) = s\}|}
      {d\;|\cP_i(X)-\delta|}.
\end{align*}
On the other hand, 
\begin{equation}
\label{excdefi}
M^i(E_\delta)(s) 
   = \sum_{r \ge 0}m_{r,s}^{(i)}\frac{|\{\omega\in\Omega_\delta : \rk(\omega) = r\}|}{|\Omega_\delta|}
   = \frac{1}{|\Omega_\delta|}\sum_{\omega\in\Omega_\delta} m_{\rk(\omega),s}^{(i)}.
\end{equation}
Using the inequality \eqref{cdefi} we conclude that
$$
\left|E_{\{\delta\} * \cP_i(X)}(s) - M^i(E_\delta)(s)\right| \le 1/Y.
$$
If $s > b$, then $E_{\{\delta\} * \cP_i(X)}(s) = 0$, and by \eqref{conv} we have 
$m_{\rk(\omega),s}^{(i)} = 0$ for every $\omega\in\Omega_\delta$.  Therefore by \eqref{excdefi} 
$M^i(E_\delta)(s) = 0$ as well.  The proposition follows.
\end{proof}

\begin{thm}
\label{ybox}
Suppose that there are constants $b_0$, $b_1$ such that for every $\delta\in\D$ 
and every $\omega\in\Omega_\delta$,
$$
\rk(\omega) \le b_1 w(\delta) + b_0.
$$
Let $\delta_0 = \emptyset \in \D$.  
Then for every $m, k \ge 0$ such that $\cup_X \D_{m,k,X}$ is nonempty, 
$$
\lim_{X \to \infty} E_{\D_{m,k,X}} = M^k(E_{\delta_0}).
$$
\end{thm}

Before proving Theorem \ref{ybox}, we have the following elementary lemma.

\begin{lem}
\label{Eavg}
If $B \subset B'$ are nonempty finite subsets of $\D$ and all $\delta \in B'$ 
have the same cardinality, then 
$$
\| E_{B} - E_{B'}\| \le 2\frac{|B'-B|}{|B|}.
$$
\end{lem}

\begin{proof}
Let $F = \sum_{\delta\in B}E_\delta \in \ell^1$ and $G =\sum_{\delta\in B'-B}E_\delta \in \ell^1$.  Then
$$
E_{B} - E_{B'} = \frac{F}{|B|} - \frac{F+G}{|B'|} = \frac{(|B'|-|B|)F - |B|G}{|B||B'|}
$$
so 
$$
\|E_{B} - E_{B'}\| \le \frac{|B'-B|}{|B|} \frac{\| F \|}{|B|} + \frac{\| G \|}{|B'|} 
   \le \frac{|B'-B|}{|B|} + \frac{|B'-B|}{|B|}.
$$
\end{proof}

\begin{proof}[Proof of Theorem \ref{ybox}]
We will prove this by induction on $m$.  If $m=0$, then $k=0$, $\D_{m,k,X} = \{\delta_0\}$ for every $X$, 
and there is nothing to prove.

Now suppose $m \ge 1$.  Define 
$$
\D_{m,k,X}' := \{\delta\in\D_{m,k,X} : \text{$\N(q) \le L_{m-1}(X)$ for every $q \in \delta$}\}.
$$
and for every $i \in I$, let $\cP_i(X,Y) := \{q \in \cP_i : X \le \N(q) < Y\}$ and 
$$
B_{i,X} := \D_{m-1,k-i,X} * \cP_i(L_{m-1}(X),L_m(X)).
$$  
Then 
\begin{equation}
\label{prelim}
\D_{m,k,X} = \coprod_{i \in I} B_{i,X} \coprod \D_{m,k,X}'
\end{equation}

If $\delta \in \D_{m-1,k-i}$ then Lemma \ref{Eavg} and \eqref{growth} show that for large $X$,
\begin{equation}
\label{2/X}
\| E_{\{\delta\}*\cP_i(L_m(X))} - E_{\{\delta\}*\cP_i(L_{m-1}(X),L_{m}(X))} \| 
   \le \frac{2 \,|\cP_i(L_{m-1}(X))|}{|\cP_i(L_{m-1}(X),L_{m}(X))|}.
\end{equation}
Suppose $\D_{m-1,k-i,X}$ is nonempty, and abbreviate $D_X := \D_{m-1,k-i,X}$.  
We will apply Proposition \ref{-1} with $Y = \prod_{j < m} L_j(X)$.
For every $\delta \in $ we have $\N(\delta) \le Y$, and $L_m(X) \ge \Ceb(Y)$. 
Thus by \eqref{2/X} and Proposition \ref{-1}
\begin{align*}
\| E_{B_{i,X}} - M^i&(E_{D_X}) \|
   = \left\| \frac{\sum_{\delta\in D_X}E_{\{\delta\}*\cP_i(L_{m-1}(X),L_m(X))}}{|D_X|} 
      - \frac{\sum_{\delta\in D_X}M^i(E_\delta)}{|D_X|} \right\| \\ 
   &\le \frac{\sum_{\delta\in D_X}\| E_{\{\delta\}*\cP_i(L_m(X))} - M^i(E_\delta)\|}{|D_X|} 
      + \frac{2 \,|\cP_i(L_{m-1}(X))|}{|\cP_i(L_{m-1}(X),L_{m}(X))|} \\
   &\le \frac{b_1 k + b_0 + 1}{\prod_{j < m} L_j(X)} 
      + \frac{2 \,|\cP_i(L_{m-1}(X))|}{|\cP_i(L_{m-1}(X),L_{m}(X))|}.
\end{align*}
Both terms go to zero as $X$ grows (using \eqref{growth} 
for the second term), and by our induction hypothesis
$\lim_{X\to\infty}E_{D_X} = M^{k-i}(E_{\delta_0})$, so for every $i \in I$
\begin{equation}
\label{lim}
\lim_{X \to \infty} E_{B_{i,X}} = M^k(E_{\delta_0}).
\end{equation}

By \eqref{growth} we see that for every $\epsilon > 0$, as $X$ grows we have
$$
|\D_{m,k,X}'| \ll \bigl(L_{m-1}(X)\prod_{j < m} L_j(X) \bigr)^{1+\epsilon}
$$
and either $B_{i,X}$ is empty or
$$
|B_{i,X}| \gg \bigl(\prod_{j \le m} L_j(X)\bigr)^{1-\epsilon}. 
$$
In particular 
$
\lim_{X\to\infty} |\D_{m,k,X}'|/\sum_i |B_{i,X}| = 0,
$  
so by Lemma \ref{Eavg} and equations \eqref{prelim} and \eqref{lim},
$$
\lim_{X \to \infty} E_{\D_{m,k,X}} 
   = \lim_{X \to \infty} E_{\coprod B_{i,X}}
   = M^k(E_{\delta_0}).
$$
\end{proof}

\begin{defn}
\label{3.10}
Let $\D^{(k)}_X = \cup_m \D_{m,k,X}$.
\end{defn}

Note that $\cup_X\D_{m,k,X}$ is nonempty if and only if $k$ can be written as 
a sum of $m$ (not necessarily distinct) elements of $I$.  In particular, if $\cup_X\D_{m,k,X}$ 
is nonempty then $m \le k$, so $\D^{(k)}_X$ is finite for 
every $k$.

\begin{cor}
\label{4.6}
Suppose that the hypotheses of Theorem \ref{ybox} hold, and $M = \ML$, the mod $p$ 
Lagrangian operator of Definition \ref{MLdef}.  Then 
\begin{align*}
\lim_{k \to \infty}\lim_{X \to \infty} E_{\D^{(2k)}_X} 
   &= (1-\parity(E_{\delta_0}))\Peven + \parity(E_{\delta_0})\Podd, \\
\lim_{k \to \infty}\lim_{X \to \infty} E_{\D^{(2k+1)}_X} 
   &= \parity(E_{\delta_0})\Peven + (1-\parity(E_{\delta_0}))\Podd.
\end{align*}
where $\Peven$ and $\Podd$ are given by Definition \ref{10.3}.
In particular these limits depend only on the parity $\parity(E_{\delta_0})$ of the initial state $E_{\delta_0}$.
If $\parity(E_{\delta_0}) = 1/2$, then 
$$
\textstyle
\lim_{k \to \infty}\lim_{X \to \infty} E_{\D^{(k)}_X} 
   = \frac{1}{2}\Peven + \frac{1}{2}\Podd.
$$
\end{cor}

\begin{proof}
This follows directly from Theorem \ref{ybox} and Proposition \ref{10.5}.
\end{proof}

\part{Application to the distribution of Selmer ranks}

\section{Setup}
\label{notation}

For the rest of this paper we will apply the results of Part \ref{part1} 
to study the distribution of Selmer ranks in families of twists.

Fix a number field $K$ and a rational prime $p$.  Let $\Kb$ denote a fixed algebraic 
closure of $K$, and $G_K := \Gal(\Kb/K)$.  Let $\bmu_p$ denote the group of $p$-th roots 
of unity in $\Kb$.  
We will use $v$ (resp., $\l$) for a place (resp., nonarchimedean place, or prime ideal) of $K$.
If $v$ is a place of $K$, we let $K_v$ denote the completion of $K$ at $v$, and 
$K_v^\ur$ its maximal unramified extension.  

Fix also a two-dimensional 
$\Fp$-vector space $T$ with a continuous action of $G_K$, and with a nondegenerate 
$G_K$-equivariant alternating pairing corresponding to an isomorphism
\begin{equation}
\label{weilpair}
\wedge^2T \isom \bmu_p.
\end{equation}
We say that $T$ is unramified at $v$ if the inertia subgroup of $G_{K_v}$ 
acts trivially on $T$, and in that case we define the unramified subgroup 
$\Hu(K_v,T) \subset H^1(K_v,T)$ by
$$
\Hu(K_v,T) := H^1(K_v^\ur/K_v,T) = \ker [H^1(K_v,T) \to H^1(K_v^\ur,T)].
$$
If $c \in H^1(K,T)$ and $v$ is a place of $K$, we will often abbreviate 
$c_v := \loc_v(c)$ for the localization of $c$ in $H^1(K_v,T)$.

We also fix a finite set $\Sigma$ of places of $K$, containing all places where $T$ is ramified, all 
primes above $p$, and all archimedean places.

\begin{defn}
\label{qfdef}
If $V$ is a vector space over $\Fp$, a {\em quadratic form} on $V$ is a function $q : V \to \Fp$ such that 
\begin{itemize}
\item
$q(av) = a^2 q(v)$ for every $a \in \Fp$ and $v \in V$,
\item
the map $(v,w)_q := q(v+w)-q(v)-q(w)$ is a bilinear form.
\end{itemize}
If $X \subset V$, we denote by $X^\perp$ the orthogonal complement of $X$ in $V$ under 
the pairing $(\;\;,\,\;)_q$.  We say that $(V,q)$ is a {\em \metabolic space} if 
$(\;\;,\,\;)_q$ is nondegenerate and $V$ has a subspace 
$X$ such that $X = X^\perp$ and $q(X) = 0$.  Such a subspace $X$ is called 
a {\em Lagrangian subspace} of $V$.
\end{defn}

For every place $v$ of $K$, the cup product and the pairing \eqref{weilpair} 
induce a pairing
$$
H^1(K_v,T) \times H^1(K_v,T) \map{\;\cup\;} H^2(K_v,T \otimes T) \too H^2(K_v,\bmu_p).
$$
For every $v$ there is a canonical inclusion $H^2(K_v,\bmu_p) \hookto \Fp$
that is an isomorphism if $v$ is nonarchimedean.
The local Tate pairing is the composition
\begin{equation}
\label{tatepair}
\ld \;\;,\;\rd_v : H^1(K_v,T) \times H^1(K_v,T) \too \Fp.
\end{equation}

\begin{defn}
Suppose $v$ is a place of $K$.  
We say that $q$ is a {\em Tate quadratic form} on $H^1(K_v,T)$ if the bilinear form 
induced by $q$ (Definition \ref{qfdef}) is $\ld \;\;,\;\, \rd_v$.  
If $v \notin \Sigma$, then 
we say that $q$ is {\em unramified} if $q(x) = 0$ for all $x \in \Hu(K_v,T)$.
\end{defn}

\begin{defn}
\label{ahsdef}
Suppose $T$ is as above.  A {\em global \metabolic structure} $\bq$ on $T$ consists 
of a Tate quadratic form $q_v$ on $H^1(K_v,T)$ for every place $v$, such that
\begin{enumerate}
\item
$(H^1(K_v,T),q_v)$ is a \metabolic space for every $v$,
\item
if $v \notin \Sigma$ then $q_v$ is unramified,
\item
if $c \in H^1(K,T)$ then $\sum_v q_v(c_v) = 0$.
\end{enumerate} 
\end{defn}

Note that if $c \in H^1(K,T)$ then $c_v \in \Hu(K_v,T)$ for almost all $v$, 
so the sum in Definition \ref{ahsdef}(iii) is finite.

\begin{defn}
Suppose $v$ is a place of $K$ and $q_v$ is a quadratic form on $H^1(K_v,T)$.  
Let 
$$
\H(q_v) := \{\text{Lagrangian subspaces of $(H^1(K_v,T),q_v)$}\},
$$
and if $v \notin \Sigma$
$$
\H_\ram(q_v) := \{X \in \H(q_v) : X \cap \Hu(K_v,T) = 0\}.
$$
\end{defn}

\begin{lem}
\label{countlem}
Suppose $v \notin \Sigma$ and $q_v$ is a Tate quadratic form on $H^1(K_v,T)$.  
Let $d_v := \dim_{\Fp}T^{G_{K_v}}$.  
Then:
\begin{enumerate}
\item
$\dim_{\Fp}H^1(K_v,T) = 2d_v$, 
\item
every $X \in \H(q_v)$ has dimension $d_v$, 
\item
if $d_v > 0$ and $q_v$ is unramified, then 
$|\H_\ram(q_v)| = p^{d_v-1}$.
\end{enumerate}
\end{lem}

\begin{proof}
\cite[Lemma 3.7]{kmr} (Assertion (iii) follows from \cite[Proposition 2.6]{poonenrains}.)
\end{proof}

\begin{defn}
\label{ssdef}
Suppose $T$ is as above and $\bq$ is a global \metabolic structure on $T$.
A {\em Selmer structure} $\cS$ for $(T,\bq)$ (or simply for $T$, if $\bq$ is understood) 
consists of 
\begin{itemize}
\item
a finite set $\Sigma_\cS$ of places of $K$, containing $\Sigma$,
\item
for every $v \in \Sigma_\cS$, a Lagrangian subspace $\HF(K_v,T) \subset H^1(K_v,T)$.
\end{itemize}
If $\cS$ is a Selmer structure, we set $\HF(K_v,T) := \Hu(K_v,T)$ if $v \notin \Sigma_\cS$, 
and we define the {\em Selmer group} $\HF(K,T) \subset H^1(K,T)$ 
by 
$$
\HF(K,T) := \ker (H^1(K,T) \too \dirsum{v}H^1(K_v,T)/\HF(K_v,T)),
$$
i.e., the subgroup of $c \in H^1(K,T)$ such that $c_v \in \HF(K_v,T)$ for every $v$.
\end{defn}

\begin{defn}
\label{setdef}
If $L$ is a field, define
$$
\Xset(L) := \Hom(G_L,\bmu_p)
$$
(throughout this paper, ``$\Hom$'' will always mean continuous homomorphisms).
If $L$ is a local field, we let $\Xset_\ram(L) \subset \Xset(L)$ denote the 
subset of ramified characters.  In this case local class field theory identifies 
$\Xset(L)$ with $\Hom(L^\times,\bmu_p)$, and $\Xset_\ram(L)$ is then the subset of 
characters nontrivial on the local units $\O_L^\times$.
Let $\one_L \in \Xset(L)$ denote the trivial character.

There is a natural action of $\Aut(\bmu_p) = \Fp^\times$ on $\Xset(L)$, and we let 
$\Fset(L) := \Xset(L)/\Aut(\bmu_p)$.  Then $\Fset(L)$ is naturally identified with 
the set of cyclic extensions of $L$ of degree dividing $p$, via the correspondence 
that sends $\chi \in \Xset(L)$ to the fixed field $\bar{L}^{\ker(\chi)}$ of $\ker(\chi)$ 
in $\bar{L}$.  If $L$ is a local field, then $\Fset_\ram(L)$ denotes the set of ramified 
extensions in $\Fset(L)$.
\end{defn}
  
\begin{defn}
Define
\begin{align*}
\cP_{i\phantom{0}} &:=\; \{\l : \text{$\l \notin \Sigma$, $\bmu_p \subset K_\l$, 
   and $\dim_{\Fp}T^{G_{K_\l}} = i$}\}\quad \text{if $1 \le i \le 2$}, \\
\cP_{0\phantom{i}} &:=\; \{\l : \l \notin \Sigma \cup \cP_1 \cup \cP_2\}, \\
\cP_{\phantom{i0}} &:=\; \cP_0 \textstyle\coprod \cP_1 \coprod \cP_2  = \{\l : \l \notin \Sigma\}.
\end{align*}
Define the {\em width} function $w : \cP \to \{0,1,2\}$ by
$w(\l) := i$ if $\l\in\cP_i$.
\end{defn}

Let $K(T)$ denote the field of definition of the elements of $T$, i.e., 
the fixed field in $\Kb$ of $\ker (G_K \to \Aut(T))$.

\begin{lem}
\label{4.2}
Suppose $\l$ is a prime of $K$, $\l \notin \Sigma$, and let $\Frob_\l \in \Gal(K(T)/K)$ 
be a Frobenius element for some choice of prime above $\l$.  Then
\begin{enumerate}
\item
$\l\in\cP_2$ if and only if $\Frob_\l = 1$,
\item
$\l\in\cP_1$ if and only if $\Frob_\l$ has order exactly $p$,
\item
$\l\in\cP_0$ if and only if $\Frob_\l^p \ne 1$.
\end{enumerate}
\end{lem}

\begin{proof}
\cite[Lemma 4.3]{kmr}
\end{proof}

\begin{defn}
\label{twistdata}
Suppose $T$, $\Sigma$ are as above, and $\bq$ is a global \metabolic structure on $T$.
By {\em twisting data} we mean
\begin{enumerate}
\item
for every $v \in \Sigma$, a (set) map 
$$
\alpha_v : \Xset(K_v)/\Aut(\bmu_p) = \Fset(K_v) \too \H(q_v),
$$
\item
for every $v \in \cP_2$, a bijection 
$$
\alpha_v : \Xset_\ram(K_v)/\Aut(\bmu_p) = \Fset_\ram(K_v) \too \H_\ram(q_v).
$$
\end{enumerate}
\end{defn}

\begin{defn}
\label{ddef}
Let
$$
\D := \{\text{squarefree products of primes $\l \in \cP_1 \cup \cP_2$}\},
$$
and if $\d\in\D$ let $\d_1$ (resp., $\d_2$) be the product of all primes 
dividing $\d$ that lie in $\cP_1$ (resp.,  $\cP_2$), so $\d = \d_1\d_2$.
For every $\d \in \D$, define also 
\begin{itemize}
\itemsep=5pt
\item
$
w(\d) := \sum_{\l\mid\d} w(\l) = |\{\l : \l\mid\d_1\}| + 2\cdot|\{\l : \l\mid\d_2\}|,
$
the {\em width} of $\d$,
\item
$\Sigma(\d) := \Sigma \cup \{\l : \l \mid \d\} \subset \Sigma \cup \cP_1 \cup \cP_2$,
\item
$
\Omega_\d := \prod_{v \in \Sigma}\Xset(K_v) \;\times 
   \prod_{\l \mid \d}\Xset_\ram(K_\l),
$
\item
$\Omega_\d^S := S \;\times \prod_{\l \mid \d}\Xset_\ram(K_\l)$ for every subset $S \subset \Omega_1 = \prod_{v \in \Sigma}\Xset(K_v)$,
\item
$\eta_{\d,\l} : \Omega^S_{\d\l} \to \Omega^S_{\d}$ the projection map, if $\d\l\in\D$.
\end{itemize}
\end{defn}

Note that $\D$ can be identified with the set of finite subsets of $\cP_1 \cup \cP_2$, 
as in \S\ref{avgs}.\ref{1.C}.
  
\begin{defn}
\label{sstwist}
Given $T$, $\bq$, and twisting data as in Definition \ref{twistdata}, we define 
a Selmer structure $\cS(\omega)$ for every $\d\in\D$ and 
$\omega = (\omega_v)_v \in\Omega_\d$ as follows.
\begin{itemize}
\item
Let $\Sigma_{\cS(\omega)} := \Sigma(\d)$.
\item
If $v \in \Sigma$ then let $H^1_{\cS(\omega)}(K_v,T) := \alpha_v(\omega_v)$,
\item
If $v \mid \d_1$, let $H^1_{\cS(\omega)}(K_v,T)$ be the unique element of $\H_\ram(q_v)$.
\item
If $v \mid \d_2$, let $H^1_{\cS(\omega)}(K_v,T) := \alpha_v(\omega_v) \in \H_\ram(q_v)$.
\end{itemize}
If $\omega\in\Omega_\d$ we will also write $\Sel(T,\omega) := H^1_{\cS(\omega)}(K,T)$.
\end{defn}

\begin{thm}
\label{kramer}
Suppose $\d \in \D$, $\omega \in \Omega_1$, and $\omega' \in \Omega_\d$.  Then
\begin{multline*}
\dim_{\F_p}\Sel(T,\omega) - \dim_{\F_p}\Sel(T,\omega')  \\\equiv 
   w(\d) + \sum_{v \in \Sigma}\dim_{\Fp}\alpha_v(\omega_v)/(\alpha_v(\omega_v)\cap\alpha_v(\omega_v')) \pmod{2}.
\end{multline*}
\end{thm}

\begin{proof}
\cite[Theorem 4.11]{kmr}
\end{proof}

\begin{rem}
By Lemma \ref{4.2} and the Cebotarev theorem, $\cP_2$ is a normed set with linear growth in the 
sense of Definition \ref{nsg}, and the same holds for $\cP_1$ if $p \mid [K(T):K]$.
(If $p \nmid [K(T):K]$ then Lemma \ref{4.2} shows that $\cP_1$ is empty.)

If $\d\in\D$ and $\omega\in\Omega_\d$, define $\rk(\omega) := \dim_{\Fp}\Sel(T,\omega)$.
For every choice of subset $S \subset \Omega_1$, 
the sets $\{\Omega^S_\d : \d\in\D\}$, together with the functions $\rk : \Omega^S_\d \to \Z_{\ge 0}$ 
and $\eta_{\d,\l}$, give rank data on $\D$ as in Definition \ref{rsdef} 
(using Proposition \ref{basic}(i) below).

We will show in \S\ref{csr} below that the rank data $\Omega^S$ is governed 
(in the sense of Definition \ref{govdef}) by 
the mod $p$ Lagrangian Markov operator $\ML$ of Definition \ref{MLdef}.  
We will then be able to apply Theorem \ref{ybox}.
\end{rem}

\section{Example: twists of elliptic curves}
\label{eex}

Fix for this section an elliptic curve $A$ defined over $K$, a prime $p$, 
and let $T := A[p]$.  We will show that this $T$ 
comes equipped with the extra structure that we require, and that with an appropriate 
choice of twisting data, the Selmer groups $\Sel(A[p],\chi)$ are classical $p$-Selmer groups 
of twists of $A$.

The module $T = A[p]$ satisfies the hypotheses of \S\ref{notation}, with the pairing 
\eqref{weilpair} given by the Weil pairing. Let $\Sigma$ be a finite set of places 
of $K$ containing all archimedean places, all places above $p$, and all primes where 
$A$ has bad reduction.
Let $\O$ denote the ring of integers of the cyclotomic field of $p$-th roots of unity, 
and $\p$ the (unique) prime of $\O$ above $p$.

If $p > 2$, there is a unique global \metabolic structure $\bq_A = (q_{A,v})$ on $A[p]$.  
For general $p$, there is a canonical 
global \metabolic structure $\bq_A$ on $A[p]$ constructed from the Heisenberg group, 
see \cite[\S4]{poonenrains} or the proof of \cite[Lemma 5.2]{kmr}.

We next define twisting data for $(A[p],\Sigma,\bq_A)$ in the sense of Definition \ref{twistdata}.

\begin{defn}
\label{etwist}
Suppose $\chi\in\Xset(K)$ (or $\chi \in \Xset(K_v)$) is nontrivial.  
If $p=2$ we let $A^\chi$ denote the quadratic twist of $A$ by $\chi$ over $K$ (resp., $K_v$).
For general $p$, let $F$ denote the cyclic
extension of $K$ (resp., $K_v$) of degree $p$ corresponding to $\chi$, 
and let $A^\chi$ denote the abelian variety 
denoted $A_F$ in \cite[Definition 5.1]{mrs}.  

Concretely, if $\chi \in \Xset(K)$ and $\chi\ne \one_K$ then $A^\chi$ is an abelian variety of 
dimension $p-1$ over $K$, defined to be the kernel of the canonical map
$$
\Res^F_K(A) \too A
$$
where $\Res^F_K(A)$ denotes the Weil restriction of scalars of $A$ from $F$ to $K$.
The character $\chi$ induces an inclusion $\O \subset \End_K(A^\chi)$ (see \cite[Theorem 5.5(iv)]{mrs}).  
If $\pi$ is a generator of the ideal $\p$ of $\O$, then we denote by 
$\Sel_\pi(A^\chi/K)$ the usual $\pi$-Selmer group of $A^\chi/K$.
In particular when $p=2$, $\Sel(A[2],\chi) = \Sel_2(A^\chi/K)$ is 
the classical $2$-Selmer group of $A^\chi/K$.

For $\chi \in \Xset(K)$, let $\bq_{A^\chi} = (q_{A^\chi,v})$ be the unique global \metabolic structure 
on $A^\chi[\p]$ if $p>2$, and if $p=2$ 
we let $\bq_{A^\chi}$ be the canonical global \metabolic structure on the elliptic curve $A^\chi$.
\end{defn}

If $p=2$, then the two definitions above of $A^\chi$ agree, with $\O = \Z$, and $\p = 2$.

\begin{lem}
\label{sameT}
There is a canonical $G_K$-isomorphism $A^\chi[\p] \cong A[p]$, which 
identifies $q_{A^\chi,v}$ with $q_{A,v}$ for every $v$ 
and every $\chi\in\Xset(K_v)$.
\end{lem}

\begin{proof}
\cite[Lemma 5.2]{kmr}
\end{proof}

\begin{defn}
\label{etwda}
Let $\pi$ denote any generator of the ideal $\p$ of $\O$.
If $v$ is a place of $K$ and $\chi \in \Xset(K_v)$, define $\alpha_v(\chi)$ to be the image of the 
composition of the Kummer ``division by $\pi$'' map with the isomorphism of Lemma \ref{sameT}(i)
$$
\alpha_v(\chi) := \image\biggl(A^\chi(K_v)/\p A^\chi(K_v) \hookto H^1(K_v,A^\chi[\p]) 
   \isom H^1(K_v,A[p])\biggr).
$$
\end{defn}

Note that $\alpha_v(\chi)$ is independent of the choice of generator $\pi$.
For every place $v$ and $\chi\in\Xset(K_v)$, \cite[Lemma 5.4]{kmr} shows that $\alpha_v(\chi) \in \H(q_{A,v})$.

\begin{prop}
\label{6.4}
\begin{enumerate}
\item
The maps $\alpha_v$ of Definition \ref{etwda}, for $v \in \Sigma$ and $v \in \cP_2$, 
give twisting data as in Definition \ref{twistdata}.
\item
Suppose $\chi \in \Xset(K)$, and let $\d$ be the part of the conductor of $\chi$ 
supported on $\cP_1 \cup \cP_2$.
With the twisting data of (i), and any generator $\pi$ of $\p$, we have 
$$
\Sel_\pi(A^\chi/K) \cong \Sel(A[p],\omega)
$$
where $\omega = (\ldots,\chi_v,\ldots)_{v \in\Sigma(\d)} \in \Omega_\d$ with $\chi_v \in \Xset(K_v)$ the 
restriction of $\chi$ to $G_{K_v}$.
\end{enumerate}
\end{prop}

\begin{proof}
\cite[Propositions 5.8 and 5.9]{kmr}
\end{proof}

\section{Changing Selmer ranks}
\label{csr}

In this section we study how the Selmer rank changes when we change one local condition, 
i.e., we study $\dim_{\Fp}\Sel(T,\omega) - \dim_{\Fp}\Sel(T,\bar\omega)$ when 
$\omega\in\Omega_{\d\l}$ projects to $\bar\omega\in\Omega_\d$.
Proposition \ref{basic} evaluates this difference in terms of the 
dimension of the localization $\loc_\l(\Sel(T,\bar\omega))$, and Proposition \ref{probs} describes 
the distribution of the values $\dim_{\Fp}\loc_\l(\Sel(T,\bar\omega))$ as $\l$ varies.

For the rest of this paper we
fix $T$ and $\Sigma$ as in \S\ref{notation}, a global \metabolic structure $\bq$ on $T$ 
as in Definition \ref{ahsdef}, 
and twisting data as in Definition \ref{twistdata}.  
Recall that $K(T)$ is the field of definition of the elements of $T$, i.e., 
the fixed field in $\Kb$ of $\ker (G_K \to \Aut(T))$.

For the rest of this paper we assume also that
\begin{equation}
\label{h2a} 
\Pic(\O_{K,\Sigma}) = 0,
\end{equation}
and
\begin{equation}
\label{h2b}
\O_{K,\Sigma}^\times/(\O_{K,\Sigma}^\times)^p \too \prod_{v \in \Sigma} K_v^\times/(K_v^\times)^p \quad
   \text{is injective},
\end{equation}
where $\O_{K,\Sigma}$ is the ring of $\Sigma$-integers of $K$, i.e., the 
elements that are integral at all $\l \notin \Sigma$.
Lemma 6.1 of \cite{kmr} shows that \eqref{h2a} and \eqref{h2b} can always be satisfied by 
enlarging $\Sigma$ if necessary.

Recall the set $\D$, and for $\d\in\D$ the sets $\Sigma(\d)$, $\Omega_\d$, and 
$\Xset(\d)$, all from Definition \ref{ddef}.  
If $\d \in \D$ and $\omega \in \Omega_\d$, recall that $\rk(\omega) := \dim_{\Fp}\Sel(T,\omega)$, 
and if $\d\l \in\D$, let $\eta_{\d,\l} : \Omega_{\d\l} \to \Omega_\d$ be the natural projection.

\begin{prop}
\label{basic}
Suppose $\d \in \D$, $\bar\omega \in \Omega_\d$, and $\l \in \cP_1 \cup \cP_2$ and $\l\nmid\d$.  Let 
$$
t(\l) = t(\bar\omega,\l) := \dim_{\Fp} \image(\Sel(T,\bar\omega) \map{\loc_\l} H^1_\ur(K_\l,T)).
$$
\begin{enumerate}
\item
We have $|\eta_{\d,\l}^{-1}(\bar\omega)| = p(p-1)$.
\item
Suppose $\l \in \cP_1$ and $\omega \in \eta_{\d,\l}^{-1}(\bar\omega) \subset \Omega_{\d\l}$. 
Then $0 \le t(\l) \le 1$, and
$$
\rk(\omega) = 
\begin{cases}
\rk(\bar\omega) - 1 & \text{if $t(\l) = 1$,}\\
\rk(\bar\omega) + 1 & \text{if $t(\l) = 0$.}
\end{cases}
$$
\item
Suppose $\l \in \cP_2$.  Then $0 \le t(\l) \le 2$, and 
$$
\rk(\omega) = 
\begin{cases}
\rk(\bar\omega) - 2 & \text{if $t(\l) = 2$, for every $\omega\in \eta_{\d,\l}^{-1}(\bar\omega)$}\\
\rk(\bar\omega) & \text{if $t(\l) = 1$, for every $\omega\in \eta_{\d,\l}^{-1}(\bar\omega)$,}\\
\rk(\bar\omega)+2 & \text{if $t(\l) = 0$, for exactly $p-1$ of the $\omega\in\eta_{\d,\l}^{-1}(\bar\omega)$,}\\ 
\rk(\bar\omega) & \text{if $t(\l) = 0$, for all other $\omega\in\eta_{\d,\l}^{-1}(\bar\omega)$}.
\end{cases}
$$
\end{enumerate}
\end{prop}

\begin{proof}
For the first assertion we have $|\eta_{\d,\l}^{-1}(\bar\omega)| = |\Xset_\ram(K_\l)| = p(p-1)$.

Let $\cS(\bar\omega)$ be the Selmer structure of Definition \ref{sstwist}.
Define 
\begin{align*}
\Sel(T,\bar{\omega})^{(\l)} &:= \ker (H^1(K,T) \map{\oplus \loc_v} \dirsum{v\ne\l}H^1(K_v,T)/H^1_{\cS(\bar{\omega})}(K_v,T)), \\
\Sel(T,\bar{\omega})_{(\l)} &:= \ker (H^1_{\cS(\bar{\omega})}(K_v,T)^{(\l)} \map{\loc_\l} H^1(K_\l,T)).
\end{align*}
Then we have  
$\Sel(T,\bar{\omega})_{(\l)} \subset \Sel(T,\bar{\omega}) \subset \Sel(T,\bar{\omega})^{(\l)}$, and 
if $\omega \in \eta_{\d,\l}^{-1}(\bar\omega)$ then 
$\Sel(T,\bar{\omega})_{(\l)} \subset \Sel(T,{\omega}) \subset \Sel(T,\bar{\omega})^{(\l)}$ as well.

Let $V := \loc_\l(\Sel(T,\bar{\omega})^{(\l)}) \subset H^1(K_\l,T)$.  Poitou-Tate global duality 
(see for example \cite[Theorem I.4.10]{milne} or \cite[Theorem 3.1]{tate}) shows that 
$V$ is a maximal isotropic subspace of $H^1(K_\l,T)$ with respect to the local Tate pairing, and 
by Definition \ref{ahsdef}(iii), the quadratic form $q_\l$ vanishes on $V$, so $V \in \H(q_\l)$.  
In particular if $\l\in\cP_i$, then by Lemma \ref{countlem},
$$
\textstyle
\dim_{\Fp}V = \frac{1}{2}\dim_{\Fp}H^1(K_\l,T) = i.
$$

Let $V_\ur := \Hu(K_\l,T) \in \H(q_\l)$, the unramified subspace.
Suppose that $\omega \in \eta_{\d,\l}^{-1}(\bar\omega)$, and let $\omega_\l$ be its $\l$-component.  
If $i = 1$ let $V_{\omega_\l}$ be the unique element of $\H_\ram(q_\l)$, and if $i = 2$ 
let $V_{\omega_\l} := \alpha_\l(\omega_\l)$, where $\alpha_\l : \Xset(K_\l) \to \H_\ram(q_\l)$ 
is part of the given twisting data.  Then by definition we have exact sequences
\begin{gather*}
0 \too \Sel(T,\bar{\omega})_\l \too \Sel(T,\bar{\omega}) \map{\loc_\l} V \cap V_\ur \too 0 \\
0 \too \Sel(T,\bar{\omega})_\l \too \Sel(T,{\omega}) \map{\loc_\l} V \cap V_{\omega_\l} \too 0,
\end{gather*}
and $t(\l) = \dim_{\Fp}(V \cap V_\ur)$.
We deduce that 
\begin{equation}
\label{diff}
\rk(\omega) - \rk(\bar\omega) = \dim_{\Fp}(V \cap V_{\omega_\l}) - t(\l).
\end{equation}

Suppose first that $\l\in\cP_1$, so $i=1$.  We have $V \in \H(q_\l) = \{V_\ur,V_{\omega_\l}\}$, 
and $\dim_{\Fp}(V_\ur) = \dim_{\Fp}(V_{\omega_\l}) = 1$. 
If $V = V_\ur$ then $t(\l) = 1$ and $V \cap V_{\omega_\l} = 0$, and if $V = V_{\omega_\l}$ then 
$t(\l) = 0$ and $V \cap V_{\omega_\l} = V$.  Now (ii) follows from \eqref{diff}.

Next, suppose that $\l\in\cP_2$.  By Theorem \ref{kramer} we have 
$
\rk(\omega) \equiv \rk(\bar\omega) \pmod{2},
$
and by definition $V_{\omega_\l} \cap V_\ur = 0$.

If $t(\l) = 2$, then $V = V_\ur$, so $V \cap V_{\omega_\l} = 0$ and 
$\rk(\omega) = \rk(\bar\omega) - 2$ by \eqref{diff}.  

If $t(\l) = 1$, then \eqref{diff} shows that $\dim_{\Fp}(V \cap V_{\omega_\l})$ must be odd.  Therefore 
$\dim_{\Fp}(V \cap V_{\omega_\l}) = 1$ and $\rk(\omega) = \rk(\bar\omega)$.

If $t(\l) = 0$, then $V \in \H_\ram(q_\l)$, and 
\eqref{diff} shows that $\dim_{\Fp}(V \cap V_{\omega_\l})$ must be even, so 
$\dim_{\Fp}(V \cap V_{\omega_\l}) = 0$ or $2$.  But $\dim_{\Fp}(V \cap V_{\omega_\l}) = 2$ 
if and only if $V_{\omega_\l} = V$.  Since $\alpha_\l : \Xset(K_\l)/\Aut(\bmu_p) \to \H_\ram(q_\l)$ 
is a bijection, there are exactly $p-1 = |\Aut(\bmu_p)|$ characters $\omega_\l \in \Xset(K_\l)$
such that $V_{\omega_\l} = V$.
Now the last part of (iii) follows from \eqref{diff}.
\end{proof}

\begin{cor}
\label{basiccor}
Suppose $\d \in \D$ and $\omega \in \Omega_\d$.  Then 
$$
\rk(\omega) \le w(\d) + \max\{\rk(\omega') : \omega'\in\Omega_1\}.
$$
\end{cor}

\begin{proof}
Let $\eta_1 : \Omega_\d \to \Omega_1$ be the natural projection.
By Proposition \ref{basic} and induction we have $\rk(\omega) \le \rk(\eta_1(\omega)) + w(\d)$, 
and the corollary follows.
\end{proof}

\section{An effective Cebotarev theorem}
\label{eCt}

\begin{thm}
\label{cebdata}
There is a nondecreasing function $\Ceb : [1,\infty) \to [1,\infty)$ such that for
\begin{itemize}
\item
every $Y \ge 1$,
\item
every $\d\in\D$ with $\N\d < Y$,
\item
every Galois extension $F$ of $K$ that is abelian of exponent $p$ over $K(T)$, 
and unramified outside of $\Sigma(\d)$,
\item
every pair of subsets $S, S' \subset \Gal(F/K)$ stable under conjugation, with $S$ nonempty, and 
\item
every $X > \Ceb(Y)$,
\end{itemize}
we have
$$
\biggl|\frac{|\{\l\notin\Sigma(\d) : \N\l \le X, \Frob_\l(F/K) \in S'\}|}
   {|\{\l\notin\Sigma(\d) : \N\l \le X, \Frob_\l(F/K) \in S\}|} - \frac{|S'|}{|S|}\biggr|
   \le \frac{1}{Y}
$$
(and in particular $\{\l\notin\Sigma(\d) : \N\l \le X, \Frob_\l(F/K) \in S\}$ is nonempty).
\end{thm}

\begin{proof}
This follows from standard effective versions of the 
Cebotarev theorem (see for example \cite[\S2, Theorems 2 and 4]{eff.ceb.}) 
together with the observations that
\begin{itemize}
\item
$[F:\Q]$ is bounded by $c_1 p^{c_2w(\d)}$ with constants $c_1, c_2$ depending only on $K(T)$ 
and $\Sigma$,
\item
the absolute discriminant $D_F$ of $F$ is bounded by $\N\d^{[K:\Q]}$ times 
a constant depending only on $K$ and $\Sigma$,
\item
the exceptional (Siegel) zeros of $\zeta_F(s)$ are bounded away from $1$ 
by a constant depending only on $[F:\Q]$ and $D_F$ (see for example 
\cite[Lemmas 8 and 11]{starksiegel}).
\end{itemize}
\end{proof}

\section{The governing Markov operator}
\label{gMo}

For the rest of the paper, we suppose that the image of the 
map $G_K \to \Aut(T)$ is large enough so that the following three properties 
hold:
\begin{gather}
\label{bighyp0}
\text{$T$ is a simple $G_K$-module}, \\
\label{bighyp1}
\Hom_{G_{K(\bmu_p)}}(T,T) = \Fp, \\
\label{bighyp2}
H^1(K(T)/K,T) = 0.
\end{gather}   

\begin{rem}
For example, \eqref{bighyp0}, \eqref{bighyp1}, and \eqref{bighyp2} hold 
if the image of the natural map $G_K \to \Aut(T) \cong \GL(T)$ contains $\SL(T)$ 
or the nor\-malizer of a Cartan subgroup.  If $p=2$ then these conditions hold 
if and only if $\Gal(K(T)/K) \cong S_3$.
\end{rem}

\begin{defn}
\label{fdxdef}
Suppose $\d\in\D$ and $\omega\in\Omega_\d$.  
Let $\Res_{K(T)}$ denote the composition 
\begin{equation}
\label{resn}
H^1(K,T) \too H^1(K(T),T)^{\Gal(K(T)/K)} = \Hom(G_{K(T)},T)^{\Gal(K(T)/K)}.
\end{equation}
Let $F_{\d,\omega}$ be the smallest extension of $K(T)$ 
such that for every $c \in \Sel(T,\omega)$, the homomorphism $\Res_{K(T)}c : G_{K(T)} \to T$ 
factors through $\Gal(F_{\d,\omega}/K(T))$.  In other words, $F_{\d,\omega}$ is the fixed field 
of $\cap_{c \in \Sel(T,\omega)}\ker(\Res_{K(T)}c)$.
\end{defn}

\begin{prop}
\label{cebhyp}
For every $\d\in\D$ and $\omega\in\Omega_\d$:
\begin{enumerate}
\item
There is a $\Gal(K(T)/K)$-module isomorphism $\Gal(F_{\d,\omega}/K(T)) \cong T^{\rk(\omega)}$.
\item
The map $\Res_{K(T)} : \Sel(T,\omega) \to \Hom(G_{K(T)},T)$ induces isomorphisms 
\begin{gather*}
\Sel(T,\omega) \isom \Hom(\Gal(F_{\d,\omega}/K(T)),T)^{\Gal(K(T)/K)}, \\
\Gal(F_{\d,\omega}/K(T)) \isom \Hom(\Sel(T,\omega),T)
\end{gather*}
\item
$F_{\d,\omega}/K$ is unramified outside of $\Sigma(\d)$.
\end{enumerate}
\end{prop}

\begin{proof}
Let $G := \Gal(K(T)/K)$ and $r := \rk(\omega)$.  Fix a basis $\{c_1,\ldots,c_r\}$ of 
$\Sel(T,\omega)$, and for each $i$ let $\tilde{c}_i = \Res_{K(T)}c_i \in \Hom(G_{K(T)},T)^G$.  
Then 
\begin{equation}
\label{7}
\tilde{c}_1 \times \cdots \times \tilde{c}_r : \Gal(F_{\d,\omega}/K(T)) \too T^r.
\end{equation}
is a $G$-equivariant injection.  Let $W$ be the $\Fp[G]$-module $\Gal(F_{\d,\omega}/K(T))$.  
Since $W$ is isomorphic to a $G$-invariant submodule of the semisimple module $T^r$, $W$ is also semisimple.  
If $U$ is an irreducible constituent of $W$, then $U$ is also an irreducible constituent of $T^r$, 
so $U \cong T$.  Therefore $W \cong T^j$ for some $j$.  Then $\dim_{\Fp}\Hom(W,T)^G = j$ 
by our assumption that $\Hom_{G_K}(T,T) = \Fp$.
On the other hand, since $H^1(K(T)/K,T) = 0$ by \eqref{bighyp2}, we have that \eqref{resn} 
is injective, so $\tilde{c}_1, \ldots, \tilde{c}_r$ are $\Fp$-linearly independent 
and $\dim_{\Fp}\Hom(W,T)^G \ge r$.  Therefore $j = r$, so \eqref{7} is an isomorphism and
(i) holds.  The two displayed maps of (ii) are injective by definition, and both sides 
of the first map (resp., second map) have order $p^r$ (resp., $p^{2r}$), so both 
maps are isomorphisms.

By Definition \ref{sstwist}, every $c \in \Sel(T,\omega)$ is unramified outside of $\Sigma(\d)$, 
so each $\Res_{K(T)}c$ is unramified outside of $\Sigma(\d)$, so $F_{\d,\omega}/K$ 
is unramified outside of $\Sigma(\d)$.
\end{proof}

\begin{prop}
\label{probs}
Fix $\d \in \D$, and $\omega \in \Omega_\d$.  
For every $\l \notin \Sigma(\d)$ let 
$$
t(\l) = t(\omega,\l) := \dim_{\Fp} \image(\Sel(T,\omega) \map{\loc_\l} H^1_\ur(K_\l,T))
$$
as in Proposition \ref{basic}, and let $c_{i,j}$ be given by the following table:
$$
\renewcommand{\arraystretch}{1.25}
\begin{array}{|r||c|c|c|}
\hline
& j=0 & j=1 & j=2 \\
\hline\hline
i=1 & p^{-\rk(\omega)} & 1-p^{-\rk(\omega)} & \\
\hline
i=2 & p^{-2\rk(\omega)} & (p+1)(p^{-\rk(\omega)}-p^{-2\rk(\omega)}) & 
1-(p+1) p^{-\rk(\omega)}+p^{1-2\rk(\omega)} \\
   \hline
\end{array}
$$
Then for $i = 2$ and $j = 0, 1, 2$, we have
$$
\lim_{X \to \infty} \frac{|\{\l\in \cP_i(X) : \l\nmid\d, t(\l) = j\}|}{|\{\l\in \cP_i(X) : \l\nmid\d\}|} = c_{i,j}.
$$
More precisely, if $\Ceb$ is a function satisfying Theorem \ref{cebdata}, 
then for every $Y > \N\d$ and every $X > \Ceb(Y)$ we have
$$
\biggl|\,\frac{|\{\l\in \cP_i(X) : \l\nmid\d, t(\l) = j\}|}{|\{\l\in \cP_i(X) : \l\nmid\d\}|} - c_{i,j}\,\biggr|
   \le \frac{1}{Y}. 
$$
If $p \mid [K(T):K]$ then the same is true for $i=1$, $j = 0,1$.
\end{prop}

\begin{proof}
Let $r := \rk(\omega)$, 
let $F_{\d,\omega}$ be the field of Definition \ref{fdxdef}, and 
for every $\l \notin \Sigma(\d)$ let $\Frob_\l \in \Gal(F_{\d,\omega}/K)$ denote a Frobenius 
automorphism for some choice of prime above $\l$.
We need to interpret the different values of $t(\l)$ as Frobenius conditions 
on $\l$.  By Lemma \ref{4.2}, $\l \in \cP_1$ if and only if $\Frob_\l|_{K(T)}$ has order $p$, 
and $\l \in \cP_2$ if and only if $\Frob_\l|_{K(T)} = 1$.

Suppose $\l \notin \Sigma(\d)$.  Then $\Hu(K_\l,T) \cong T/(\Frob_\l-1)T$, 
with the isomorphism given by evaluating $1$-cocycles on $\Frob_\l$ 
(see for example \cite[\S XIII.1]{serrecg}).  Thus $t(\l)$ is the $\Fp$-dimension 
of the subspace
$$
\{c(\Frob_\l) : \text{$c$ a cocycle representing a class in $\Sel(T,\omega)$}\} \subset T/(\Frob_\l-1)T.
$$
Let $\phi : \Gal(F_{\d,\omega}/K(T)) \isom \Hom(\Sel(T,\omega),T)$ be the isomorphism 
of Proposition \ref{cebhyp}(ii).

We first consider the case $\l\in\cP_2$, or equivalently $\Frob_\l \in \Gal(F_{\d,\omega}/K(T))$, 
so $T/(\Frob_\l-1)T = T$.  For $0 \le j \le 2$ let 
$$
R_j := \{f \in \Hom(\Sel(T,\omega),T) : \dim_{\Fp}\image(f) = j\}
$$
and let $S_j := \phi^{-1}(R_j) \subset \Gal(F_{\d,\omega}/K(T)) \subset \Gal(F_{\d,\omega}/K)$.
Then 
$$
t(\l) = j \iff \dim_{\Fp}\{c(\Frob_\l) : c \in \Sel(T,\omega)\} = j \iff \Frob_\l \in S_j.  
$$
Set $S' := S_j$ and $S := \Gal(F_{\d,\omega}/K(T))$.  Since $\Ceb$ satisfies 
Theorem \ref{cebdata} (and using Proposition \ref{cebhyp}(iii)), for every $X > \Ceb(Y)$
we have
$$
\biggl|\,\frac{|\{\l\in \cP_2(X), \l\nmid\d : t(\l) = j\}|}{|\{\l\in \cP_2(X) : \l\nmid\d\}|} 
   - \frac{|R_j|}{[F_{\d,\omega}:K(T)]}\,\biggr|  \le \frac{1}{Y}. 
$$
By Proposition \ref{cebhyp}(i) we have $[F_{\d,\omega}:K(T)] = p^{2r}$. 
Clearly $|R_0| = 1$.  We can decompose $R_1$ into a disjoint union, over the $p+1$ lines 
$\ell \subset T$, of the nonzero elements of $\Hom(\Sel(T,\omega),\ell)$.  Thus 
$|R_1| = (p+1)(p^r-1)$, and
$$
|R_2| = p^{2r} - |R_0| - |R_1| = p^{2r} - (p+1)(p^r-1) - 1 = p^{2r} - (p+1)p^r + p.
$$
This proves the proposition when $i = 2$.

Now suppose $p \mid [K(T):K]$, so that $\cP_1$ is nonempty.  
Suppose $\l\in\cP_1$, or equivalently $\Frob_\l|_{K(T)}$ has order $p$, 
so $T/(\Frob_\l-1)T$ has dimension $1$.  
Let 
\begin{multline*}
S' := \{g \in \Gal(F_{\d,\omega}/K) : 
    \text{$g|_{K(T)}$ has order $p$} \\ 
    \text{and $c(g) \in (g-1)T$ for every $c \in \Sel(T,\omega)$}\}
\end{multline*}
(note that $c(g)$ is well-defined in $T/(g-1)T$, independent of the choice of cocycle 
representing $c$).
Then $S'$ is closed under conjugation, and $t(\l) = 0$ if and only if $\Frob_\l \in S'$.
If we set $S :=  \{g \in \Gal(F_{\d,\omega}/K) : \text{$g|_{K(T)}$ has order $p$}\}$ 
then again since $\Ceb$ satisfies Theorem \ref{cebdata}, for every $X > \Ceb(Y)$ we have
$$
\biggl|\,\frac{|\{\l\in \cP_1(X) : t(\l) = 0\}|}{|\{\l\in \cP_1(X) : \l\nmid\d\}|} - \frac{|S'|}{|S|}\,\biggr|
   \le \frac{1}{Y}. 
$$

It remains to compute $|S'|/|S|$.  Let $U := \{g \in \Gal(K(T)/K) : \text{$g$ has order $p$}\}$.
Then $|S| = |U|[F_{\d,\omega}:K(T)] = p^{2r}|U|$.

Suppose $g \in \Gal(F_{\d,\omega}/K)$ and $g|_{K(T)}\in U$.  
Evaluation at $g$ induces a homomorphism $\lambda_g : \Sel(T,\omega) \to T/(g-1)T$, 
and we have $g \in S'$ if and only if $\lambda_g$ is identically zero.  
If $h \in \Gal(F_{\d,\omega}/K(T))$, then in $T/(g-1)T = T/(gh-1)T$ we have 
$$
\lambda_{gh}(c) = c(gh) = c(g) + gc(h) = \lambda_g(c) + c(h) \quad \text{for every $c \in \Sel(T,\omega)$}.
$$
Thus $gh \in S'$ if and only if the image of $h$ under the composition 
$$
\Gal(F_{\d,\omega}/K(T)) \map{~\phi~} \Hom(\Sel(T,\omega),T) \onto \Hom(\Sel(T,\omega),T/(g-1)T).
$$
is equal to $-\lambda_g$.  Since $\phi$ is an isomorphism, there are exactly $p^r$ such $h$.  
It follows that the restriction map $S' \to U$ is surjective, and all fibers have order 
$p^r$.  Therefore $|S'| = p^r|U|$, which proves the proposition when $i = 1, j= 0$.  
The result for $i=j=1$ follows since 
$$
\{\l\in \cP_1 : \l\nmid\d\} = \{\l\in \cP_1 : \l\nmid\d, t(\l) = 0\} 
   \coprod \{\l\in \cP_1 : \l\nmid\d, t(\l) = 1\}.
$$
\end{proof}

\begin{thm}
\label{gov}
For every subset $S \subset \Omega_1$, the rank data $\Omega^S$ on $\D$ is governed 
(in the sense of Definition \ref{govdef}) by the mod $p$ Lagrangian 
Markov operator $\ML$ of Definition \ref{MLdef}, and every function $\Ceb$ 
satisfying Theorem \ref{cebdata} is a convergence rate for $(\Omega^S,\ML)$.
\end{thm}

\begin{proof}
Fix $\d\in\D$ and $\omega \in \Omega_\d^S$, and let $r := \rk(\omega)$.  For $\l\in\cP_1\cup\cP_2$, 
$\l\nmid \d$, as in Propositions \ref{basic} and \ref{probs} we define
$$
t(\l) := \dim_{\Fp} \image(\Sel(T,\omega) \map{\loc_\l} H^1_\ur(K_\l,T)).
$$
If $X > 0$ and $\cP_i(X)$ is nonempty, define 
\begin{gather*}
F_i(X,s) := \frac{\sum_{\l\in\cP_i(X),\l\nmid\d}|\{\chi\in\eta_{\d,\l}^{-1}(\omega) : \rk(\chi) = s\}|}
      {\sum_{\l\in\cP_i(X),\l\nmid\d}|\eta_{\d,\l}^{-1}(\omega)|}, \\[5pt]
\Phi_{i,j}(\d,X) := 
   \frac{|\{\l\in \cP_i(X) : \l\nmid\d, t(\l) = j\}|}{|\{\l\in \cP_i(X) : \l\nmid\d\}|}. 
\end{gather*}

If $\cP_1$ is nonempty, i.e., $p \mid [K(T):K]$,
then Proposition \ref{basic}(i,ii) shows that 
\begin{align*}
F_1(X,s) &= 
\begin{cases}
0 & \text{if $s \ne r \pm 1$},\\[2pt]
\Phi_{1,1}(\d,X)  & \text{$s = r-1$}, \\[4pt]
\Phi_{1,0}(\d,X)  & \text{$s = r+1$}.
\end{cases} \\
\intertext{
Similarly, Proposition \ref{basic}(i,iii) shows that 
}
F_2(X,s) &= 
\begin{cases}
0 & \text{if $s \ne r$ or $r \pm 2$},\\[2pt]
\Phi_{2,2}(\d,X) & \text{$s = r-2$}, \\[4pt]
\Phi_{2,1}(\d,X) + \frac{p-1}{p}\Phi_{2,0}(\d,X) & \text{$s = r$}, \\[4pt]
\frac{1}{p} \Phi_{2,0}(\d,X) & \text{$s = r+2$}.
\end{cases}
\end{align*}
Proposition \ref{probs} computes $\lim_{X \to \infty}\Phi_{i,j}(\d,X)$ for $j \le i$, 
giving
\begin{align*}
\lim_{X \to \infty}F_1(X,s) &= 
\begin{cases}
0 & \text{if $s \ne r \pm 1$},\\
1 - p^{-r} & \text{$s = r-1$}, \\
p^{-r} & \text{$s = r+1$}.
\end{cases} \\
\lim_{X \to \infty}F_2(X,s) &= 
\begin{cases}
0 & \text{if $s \ne r$ or $r \pm 2$},\\
1 - (p+1)p^{-r} + p^{1-2r} & \text{$s = r-2$}, \\
p^{1-r} + p^{-r} - p^{1-2r} - p^{-1-2r} & \text{$s = r$}, \\
p^{-1-2r} & \text{$s = r+2$}.
\end{cases}
\end{align*}
The right-hand values above are equal to the matrix entries in $\ML$ and $\ML^2$, 
so this shows that $\ML$ governs the rank data for $\Omega^S$ for every $S$.  
Using the more precise convergence in Proposition \ref{probs} shows that 
$\Ceb$ is a convergence rate for $(\Omega^S,\ML)$.
\end{proof}

\section{Passage from global characters to semi-local characters}
\label{lgc}

We continue to assume that \eqref{bighyp0}, \eqref{bighyp1} and \eqref{bighyp2} all hold.

Theorems \ref{ybox} and \ref{gov} give us the machinery we need 
to see how Selmer ranks are distributed over the twists by
collections of local characters.  
However, we want to compute the distribution of Selmer 
ranks over twists by global characters.
In this section we use class field theory to study the map 
from global characters to collections of local characters.  
More precisely, we make the following definitions.

\begin{defn}
\label{defeta}
Recall that $\Xset(K) = \Hom(G_K,\bmu_p)$.  
If $\chi \in \Xset(K)$ and $v$ is a place of $K$, we let $\chi_v \in \Xset(K_v)$ 
denote the restriction of $\chi$ to $G_{K_v}$.
For $\d\in\D$, define 
\begin{multline*}
\Xset(\d) := \{\chi \in \Xset(K) : \text{$\chi$ is ramified at all $\l$ dividing $\d$} \\
\text{and unramified outside of $\Sigma(\d) \cup \cP_0\}$}
\end{multline*}
In other words, $\Xset(\d)$ is the fiber over $\d$ of the map $\Xset(K) \to \D$ that sends $\chi$ to 
the part of its conductor supported on $\cP_1 \cup \cP_2$, so we have $\Xset(K) = \coprod_{\d\in\D} \Xset(\d)$.
For $X > 0$ define
\begin{itemize}
\item
$\Xset(X) = \{\chi\in\Xset(K) : \text{$\chi$ is unramified outside of $\Sigma \cup \{\l : \N\l < X\}$}\}$
\item
$\Xset(\d,X) := \Xset(\d) \cap \Xset(X)$.
\end{itemize}
Let $\eta_\d : \Xset(\d) \to \Omega_\d$ be the natural map $\chi \to (\ldots,\chi_v,\ldots)_{v\in\Sigma(\d)}$, 
where $\chi_v \in \Xset(K_v)$ is the restriction of $\chi$ to $G_{K_v}$.
\end{defn}

The main result of this section is Theorem \ref{cft2}, which describes the image and fibers of the map 
$\eta_\d : \Xset(\d,X) \to \Omega_\d$.  For large $X$ this map is surjective if $p>2$ 
(its image depends on the parity of $w(\d)$ if $p=2$), and all nonempty fibers 
have the same cardinality.  Theorem \ref{cft2} will enable us to pass from averages over $\Omega_\d$ 
to averages over $\Xset(\d,X)$.

\begin{lem}
\label{biglem}
Let $G := \Gal(K(T)/K(\bmu_p))$.  
\begin{enumerate}
\item
There is a $\sigma \in G$ such that $\sigma^p \ne 1$. 
\item
If $p > 3$ then $G$ has no quotient of order $p$.
\item
If $p=3$ and $3 \mid |G|$, then $G = \SL_2(T)$. 
\end{enumerate}
\end{lem}

\begin{proof}
Fix an $\Fp$-basis of $T$, so that we can identify $\Gal(K(T)/K(\bmu_p))$ with a 
subgroup of $\SL_2(\Fp)$.

\medskip\noindent
{\em Case 1: $p \nmid |G|$.}  
Our assumption \eqref{bighyp0} implies that $G \ne 1$.  In this case any nontrivial   
$\sigma \in G$ satisfies (i), (ii) is trivial, and (iii) is vacuous.

\medskip\noindent
{\em Case 2: $G = \SL_2(\Fp)$.}  
All three assertions follow directly in this case. 

\medskip\noindent
{\em Case 3: $p \mid |G|$ and $G \ne \SL_2(\Fp)$.}  
In this case, \cite[Proposition 15]{serre1972} 
shows that $G$ is contained in a Borel subgroup of $\SL_2(\Fp)$.  It follows from our assumption
\eqref{bighyp1} that $G$ commutes only with scalar matrices in 
$M_{2 \times 2}(\Fp)$, and so there is a subgroup $H \subset \Fp^\times$, 
$H \not\subset \{\pm1\}$, such that with a suitable choice of basis
$$
G = \bigl\{\text{\scriptsize$\biggl(\begin{matrix}a & b \\ 0 & a^{-1}\end{matrix}\biggr)$} : a \in H, b \in \Fp\bigr\}.
$$
Now (i) and (ii) follow directly, and we must have $p > |H| \ge 3$ in this case.  
\end{proof}

\begin{lem}
\label{4.3}
Define the subgroup $\A \subset K^\times/(K^\times)^p$ by 
$$
\A := \ker(K^\times/(K^\times)^p \to K(T)^\times/(K(T)^\times)^p).
$$
\begin{enumerate}
\item
$\A$ is cyclic, generated by an element $\Delta \in \O_{K,\Sigma}^\times$.
\item
If $p = 2$, then $|\A| = 2$.
\item
If $p = 3$, then $|\A| = 1$ or $3$, and $\A = 1$ if $3 \nmid [K(T):K]$.
\item
If $p > 3$, then $\A = 1$.  
\end{enumerate}
\end{lem}

\begin{proof}
Assertion (i) is \cite[Lemma 6.2]{kmr}, which also showed that 
\begin{equation}
\label{A}
\A = \Hom(\Gal(K(T)/K(\bmu_p)),\bmu_p)^{\Gal(K(\bmu_p)/K)}.
\end{equation}  
Assumption \eqref{bighyp1} implies that if $p=2$, then $\Gal(K(T)/K) \cong S_3$.  
Now (ii) and (iii) follow directly from \eqref{A}.  

If $p > 3$, then (iv) follows from \eqref{A} and Lemma \ref{biglem}(ii).
\end{proof}

Fix once and for all a $\Delta \in \O_{K,\Sigma}^\times$ as in Lemma \ref{4.3}.  
Recall (Definition \ref{ddef}) that $\Omega_1 := \prod_{v \in \Sigma} \Xset(K_v)$, 
and more generally $\Omega_\d^S := S \;\times    \prod_{\l \mid \d}\Xset_\ram(K_\l)$
for $\d\in\D$ and $S \subset \Omega_1$.
For each $v$, local class field theory identifies $\Xset(K_v)$ 
with $\Hom(K_v^\times,\bmu_p)$.

\begin{lem}
\label{elem}
Suppose $G$ and $H$ are abelian groups, and $J \subset G \times H$ is a subgroup.  
Let $\pi_G$ and $\pi_H$ denote the projection maps from $G \times H$ to $G$ and $H$, respectively.
Let $J_0 := \ker(J \map{\pi_G} G/G^p)$.
\begin{enumerate}
\item
The image of the natural map $\Hom((G \times H)/J,\bmu_p) \to \Hom(H,\bmu_p)$ is 
$\Hom(H/\pi_H(J_0),\bmu_p)$.
\item

If $J/J^p \to G/G^p$ is injective, 
then $\Hom((G \times H)/J,\bmu_p) \to \Hom(H,\bmu_p)$ 
is surjective.
\end{enumerate}
\end{lem}

\begin{proof}
We have an exact sequence of $\Fp$-vector spaces 
$$
0 \too \pi_H(J_0) H^p/H^p \too H/H^p \too (G \times H)/J(G \times H)^p.
$$
Assertion (i) follows by applying $\Hom(\;\cdot\;,\bmu_p)$, and (ii) follows directly from (i).
\end{proof}

\begin{lem}
\label{cftlem}
Suppose that $\Ceb$ is a function satisfying Theorem \ref{cebdata}, $\d\in\D$,  
$\alpha \in \O^\times_{K,\Sigma(\d)}/(\O^\times_{K,\Sigma(\d)})^p$, and $\alpha \ne 1$.  
If $p > 2$, or if $p = 2$ and $\alpha \ne \Delta$, then there is a $\l \in \cP_0$ 
with $\N\l \le \Ceb(\N\d)$ 
such that $\alpha \notin (\O_\l^\times)^p$.
\end{lem}

\begin{proof}
Suppose first that $\alpha \notin \A$.  Then by definition $\alpha \notin (K(T)^\times)^p$, so 
$$
K(\bmu_p,\alpha^{1/p}) \cap K(T) = K(\bmu_p).
$$
By Lemma \ref{biglem}(i), there is a $\sigma \in \Gal(K(T)/K(\bmu_p))$ such that 
$\sigma^p \ne 1$.  Choose an element 
$\tau \in \Gal(K(T,\alpha^{1/p})/K(\bmu_p))$ such that $\tau|_{K(T)} = \sigma$ 
and $\tau|_{K(\bmu_p,\alpha^{1/p})} \ne 1$.  
By Theorem \ref{cebdata} applied with $F = K(T,\alpha^{1/p})$ 
and $S$ equal to the conjugacy class of $\tau$, 
we see that there is a prime $\l\notin\Sigma(\d)$ with $\N\l \le \Ceb(\N\d)$ whose 
Frobenius in $\Gal(K(T,\alpha^{1/p})/K)$ is in the conjugacy class of $\tau$.   
For such a prime $\l$, we have that $\l \in \cP_0$ by Lemma \ref{4.2}(iii) 
and $\alpha \notin (\O_\l^\times)^p$.

By Lemma \ref{4.3}, it remains only to consider the case $p = 3$, $3\mid[K(T):K]$, 
and $1 \ne \alpha \in \A$.  Then $K(\bmu_3,\alpha^{1/3}) \subset K(T)$, and 
$\Gal(K(T)/K(\bmu_3)) \cong \SL_2(\F_3)$ by Lemma \ref{biglem}(iii), so we can 
choose an element $\sigma \in \Gal(K(T)/K(\bmu_3))$ of order $6$.  
Applying Theorem \ref{cebdata} with $F = K(T)$ and 
$S$ equal to the conjugacy class of $\sigma$, 
we see that there is a prime $\l\notin\Sigma$ with $\N\l \le \Ceb(\N\d)$ whose 
Frobenius in $\Gal(K(T)/K)$ is in the conjugacy class of $\sigma$. 
For such a prime $\l$, we have that $\l \in \cP_0$ by Lemma \ref{4.2}(iii), 
and $\sigma$ acts nontrivially on $\alpha^{1/3} \in K(T)$, so $\alpha \notin (\O_\l^\times)^3$.
This completes the proof.
\end{proof}

\begin{defn}
\label{gpmdef2}
Define $\sgnd : \Omega_1 \to \bmu_p$ by 
$
\sgnd(\ldots,\omega_v,\ldots) := \prod_{v \in \Sigma} \omega_v(\Delta).
$
If $p=2$ define
$$
S^+ := \{\omega\in\Omega_1 : \sgnd(\omega) = 1\}, \quad S^- := \{\omega\in\Omega_1 : \sgnd(\omega) = -1\}.
$$
We will abbreviate $\Omega^{+}_\d = \Omega^{S^+}_\d$ and $\Omega^{-}_\d = \Omega^{S^-}_\d$ 
\end{defn}

Recall that $\eta_\d : \Xset(\d) \to \Omega_\d$ is the natural restriction map.

\begin{prop}
\label{cft2}
Suppose that $\d\in\D$, $\Ceb$ is a function satisfying 
Theorem \ref{cebdata}, and $X > \Ceb(\N\d)$.  
\begin{enumerate}
\item
If $p > 2$ then $\eta_\d : \Xset(\d,X) \to \Omega_\d$ is surjective.
\item
If $p = 2$ then 
$
\eta_\d(\Xset(\d,X)) = 
   \begin{cases}
   \Omega_\d^+ & \text{if $w(\d)$ is even}\\
   \Omega_\d^- & \text{if $w(\d)$ is odd}. 
   \end{cases}
$
\item
For every $\omega\in\eta_\d(\Xset(\d,X))$ we have
$$
\frac{|\{\chi\in\Xset(\d,X) : \eta_\d(\chi)=\omega\}|}{|\Xset(\d,X)|} 
   = \begin{cases}1/{|\Omega_\d|}&\text{if $p>2$},\\ 2/{|\Omega_\d|}&\text{if $p=2$}.\end{cases}
$$
\end{enumerate}
\end{prop}

\begin{proof} 
By our assumption \eqref{h2a}, we have $\Pic(\O_{K,\Sigma(\d)}) = 0$.
Thus global class field theory gives 
$$
\Xset(K) = \Hom(\iK/K^\times,\bmu_p) = \textstyle\Hom((\prod_{v \in \Sigma(\d)}K_v^\times \times 
      \prod_{\l\notin\Sigma(\d)}\O_\l^\times)/\O_{K,\Sigma(\d)}^\times,\bmu_p).
$$
Let 
\begin{align*}
Q_1 &:= \{\l : \l \in \cP_0, \N\l \le X\}, \\
Q_2 &:= \{\l : \l\in\cP_1\cup\cP_2,\l\nmid\d\} \cup \{\l : \l \in \cP_0, \N\l > X\}.
\end{align*}
We apply Lemma \ref{elem} with 
$$
\textstyle
G := \prod_{\l\in Q_1}\O_\l^\times, \quad
H := \prod_{v \in \Sigma(\d)}K_v^\times \times 
      \prod_{\l\in Q_2}\O_\l^\times, \quad
J := \O_{K,\Sigma(\d)}^\times.
$$
Note that for $\chi \in \Xset(K)$, we have 
$$
\chi \in \Xset(\d,X) \iff 
\text{$\chi_\l(\O_\l^\times) = 1$ for $\l \in Q_2$ and $\chi_\l(\O_\l^\times) \ne 1$ if $\l \mid\d$.}
$$

If $p > 2$, then combining \eqref{h2b}, Lemma \ref{cftlem}, and Lemma \ref{elem}(ii) 
we see that the restriction map
\begin{equation}
\label{surj?}
\textstyle
\Xset(K) \too \Hom(\prod_{v \in \Sigma(\d)}K_v^\times \times 
      \prod_{\l\in Q_2}\O_\l^\times,\bmu_p)
\end{equation}
is surjective.  Thus for every $\omega\in\Omega_\d$ we can find a $\chi\in\Xset(K)$, 
unramified outside of $\Sigma$, $\d$, and $Q_1$, that restricts to $\omega$.  Such a $\chi$ 
necessarily belongs to $\Xset(\d,X)$, and this shows that 
$\eta_\d : \Xset(\d,X) \to \Omega_\d$ is surjective, proving (i).

Similarly, if $p = 2$ then $\Delta \ne 1$ by Lemma \ref{4.3}(ii).
Lemma \ref{cftlem} shows that 
$\ker(J/J^2 \to G/G^2)$ 
is generated by $\Delta$, so by Lemma \ref{elem}(i) the image of \eqref{surj?} is exactly 
$\Hom((\prod_{v \in \Sigma(\d)}K_v^\times \times 
   \prod_{\l\in Q_2}\O_\l^\times)/\ld\Delta\rd,\{\pm1\})$.
By \cite[Lemma 6.5]{kmr}, $\Delta \in (\O_\l^\times)^2$ if $\l \in \cP_2$, and 
$\Delta$ generates $\O_\l^\times/(\O_\l^\times)^2$ if $\l \in \cP_1$.
It follows that for $\omega\in\Omega_\d$, we have $\omega \in \eta_\d(\Xset(\d,X))$ if and only if 
$\sgnd(\omega) = (-1)^{w(\d)}$.  This proves (ii).

If $\chi_1,\chi_2 \in \Xset(\d,X)$, then $\eta_\d(\chi_1) = \eta_\d(\chi_2)$ if and only if 
$\chi_1\chi_2^{-1} \in \Xset(1,X) \cap \ker(\eta_1)$.
Since $\Xset(\d,X)$ is stable under multiplication by the group $\Xset(1,X)$, it 
follows that all nonempty fibers of $\eta_\d : \Xset(\d,X) \to \Omega_\d$ have the same order 
$|\Xset(1,X) \cap \ker(\eta_1)|$.  This proves (iii).
\end{proof}

\section{Rank densities}
\label{rankdensities}

In this section we use Theorems \ref{ybox} and \ref{gov}, and the results of \S\ref{lgc}
to prove Theorem \ref{thmb} of the Introduction (Corollary \ref{11.11} below).  
We will deduce this from a finer result (Theorem \ref{6.5}).

Fix for this section a function $\Ceb$ satisfying Theorem \ref{cebdata}.  
By Theorem \ref{gov}, $\Ceb$ is a convergence rate function for $(\Omega,\ML)$.  
We continue to assume that \eqref{bighyp0}, \eqref{bighyp1}, 
and \eqref{bighyp2} hold.  
Recall that if $\omega\in\Omega_\d$ then $\rk(\omega) := \dim_{\Fp}\Sel(T,\omega)$.   
If $\chi \in \Xset(K)$ then $\chi \in \Xset(\d)$ for a (unique) $\d\in\D$, and we define 
$$
\Sel(T,\chi) = \Sel(T,\eta_\d(\chi))
$$
where $\eta_d : \Xset(\d) \to \Omega_\d$ is the product of restriction maps (Definition \ref{defeta}).
If $A$ is an elliptic curve over $K$ and $T = A[2]$ with the natural twisting data as in 
\S\ref{eex}, then Proposition \ref{6.4} shows that $\Sel(T,\chi) = \Sel_2(A^\chi)$, 
the classical $2$-Selmer group of the quadratic twist $A^\chi$ of $A$.

Define $\rk(\chi) := \dim_{\Fp}\Sel(T,\chi)$.

\begin{defn}
Suppose $\d\in\D$.
If $p = 2$, let $\Omega^+_\d$ and $\Omega^-_\d$ be the sets given by Definition \ref{gpmdef2}.  
To simplify the notation, define $\Omega^+_\d := \Omega^-_\d := \Omega_\d$ if $p > 2$.
Let $E^\pm_\d \in W$ be the probability distribution corresponding to $\Omega^\pm_\d$ as 
in Definition \ref{rsdef}.
\end{defn}

\begin{prop}
\label{11.2}
If $X > \Ceb(\N\d)$, then 
$$
\frac{|\{\chi\in\Xset(\d,X) : \rk(\chi) = n\}|}{|\Xset(\d,X)|} = 
\begin{cases}
E_\d^+(n) & \text{if $w(\d)$ is even} \\
E_\d^-(n) & \text{if $w(\d)$ is odd}.
\end{cases}
$$
\end{prop}

\begin{proof}
Let $\nu := (-1)^{w(\d)}$.  
Fix $X > \Ceb(\N\d)$.  By Proposition \ref{cft2}, the natural map
$
\eta_\d : \Xset(\d,X) \to \Omega_\d^\nu
$
is surjective, and all fibers have the same order.  
By definition, if $\chi \in \Xset(\d)$ then $\Sel(T,\chi) = \Sel(T,\eta_\d(\chi))$.
Therefore
$$
\frac{|\{\chi\in\Xset(\d,X) : \rk(\chi) = n\}|}{|\Xset(\d,X)|} 
   = \frac{|\{\omega\in\Omega_\d^\nu : \rk(\omega) = n\}|}{|\Omega_\d^\nu|}
   = E_\d^\nu(n).
$$
\end{proof}

\begin{lem}
\label{cft3}
Suppose $\d\in\D$.  If $m$ is the number of primes dividing $\d$, then 
for every $X > \Ceb(\N\d)$ we have $|\Xset(\d,X)| = (p-1)^m|\Xset(1,X)|$.
\end{lem}

\begin{proof}
Suppose $\d = \l_1\cdots\l_m$.  For each $j$, by Proposition \ref{cft2} 
we can fix a character $\chi_j \in \Xset(\l_j,X)$ that is (necessarily 
ramified at $\l_j$ and) 
unramified outside of $\l_j$, $\Sigma$ and $\cP_0$.  Then every $\chi\in\Xset(\d,X)$ 
can be written uniquely as a product of powers of the $\chi_j$ times 
a character in $\Xset(1,X)$, so the map
$$
(\Fp^\times)^m \times \Xset(1,X) \too \Xset(\d,X)
$$
defined by $(n_1,\ldots,n_m,\psi) \mapsto \chi_1^{n_1}\cdots\chi_m^{n_m}\psi$
is a bijection.  
\end{proof}

Use the chosen convergence rate function $\Ceb$ to define $\D_{m,k,X} \subset \D$ 
as in Definition \ref{boxdef}, for $m, k \in \Z_{\ge 0}$ and $X \in \R_{>0}$.

\begin{defn}
\label{11.7}
For $m, k \ge 0$, define
$$
 \B_{m,k,X} := \coprod_{\d\in\D_{m,k,X}}\Xset(\d,\Ceb(L_{m+1}(X))) \subset \Xset(K)
$$
with $L_{m+1}(X)$ as in Definition \ref{boxdef}.  We call the collection of 
sets of characters 
$\B_{m,k,X}$ a {\em fan structure} on $\Xset(K)$.
\end{defn}

\begin{rem}
The sets $\B_{m,k,X}$ depend on $T$ and $\Sigma$, because they depend on the sets 
$\cP_0$, $\cP_1$, and $\cP_2$.  But they do not depend on the chosen twisting data.  
Thus if we take two elliptic curves $A, A'$ with $A[p] \cong A'[p]$ as $G_K$-modules, 
and take the same $\Sigma$ and $\Ceb$ for both $A$ and $A'$, 
then the sets $\B_{m,k,X}$ are the same for $A$ and $A'$.
\end{rem}

\begin{thm}
\label{6.5}
Suppose \eqref{bighyp0}, \eqref{bighyp1}, and \eqref{bighyp2} hold.
If $m, k, n \ge 0$ and $\cup_X \D_{m,k,X}$ is nonempty, then
$$
\lim_{X \to \infty}
   \frac{|\{\chi\in \B_{m,k,X} : \rk(\chi) = n\}|}{|\B_{m,k,X}|} = 
\begin{cases}
M^k(E_1^+)(n) & \text{if $k$ is even}, \\
M^k(E_1^-)(n) & \text{if $k$ is odd}.
\end{cases}
$$
\end{thm}

\begin{proof}
Let $b_m(X) := \Ceb(L_{m+1}(X))$.  
By definition of $\B_{m,k,X}$, 
$$
\frac{|\{\chi\in \B_{m,k,X} : \rk(\chi) = n\}|}{|\B_{m,k,X}|} = 
   \frac{\sum_{\d\in\D_{m,k,X}} |\{\chi\in \Xset(\d,b_m(X)) : \rk(\chi) = n\}|}{\sum_{\d\in\D_{m,k,X}}|\Xset(\d,b_m(X))|}.
$$
By Lemma \ref{cft3}, $|\Xset(\d,b_m(X))|$ is independent of $\d\in\D_{m,k}$, so 
\begin{align*}
\frac{|\{\chi\in \B_{m,k,X} : \rk(\chi) = n\}|}{|\B_{m,k,X}|} &= 
   \frac{1}{|\D_{m,k,X}|}\sum_{\d\in\D_{m,k,X}} \frac{ |\{\chi\in \Xset(\d,b_m(X)) : \rk(\chi) = n\}|}{|\Xset(\d,b_m(X))|} \\
   &= \frac{1}{|\D_{m,k,X}|}\sum_{\d\in\D_{m,k,X}}E_\d^{(-1)^k}(n)
\end{align*}
using Proposition \ref{11.2} for the final equality.
By Theorem \ref{ybox} (using Corollary \ref{basiccor} to see that the hypotheses 
of Theorem \ref{ybox} hold), as $X$ grows this converges to $M^k(E_1^{(-1)^k})(n)$.
\end{proof}

\begin{lem} 
\hfill
\begin{enumerate}
\item
If $p \nmid [K(T):K]$, then $\cup_X \D_{m,k,X}$ is nonempty if and only $k = 2m$.
\item
If $p \mid [K(T):K]$, then $\cup_X \D_{m,k,X}$ is nonempty if and only $m \le k \le 2m$.
\end{enumerate}
\end{lem}

\begin{proof}
Recall that $\D_{m,k,X}$ consists of ideals $\d$ that are products of $m$ primes, with $w(\d) = k$.  

By Lemma \ref{4.2}, if $p \nmid [K(T):K]$ then $\cP_1$ is empty, so $w(\d)$ is twice 
the number of primes dividing $\d$.

If $p \mid [K(T):K]$, then $\cP_1$ and $\cP_2$ are both nonempty.  
So if $\d$ is a product of $m$ primes, then $m \le w(\d) \le 2m$.  Conversely,   
if $m \le k \le 2m$ then every $\d$ that is a product of $(2m-k)$ primes from $\cP_1$ 
and $(k-m)$ primes from $\cP_2$ will have $m$ prime factors and $w(\d) = k$.
\end{proof}

Recall the probability distributions $\Peven, \Podd$ given explicitly by Definition \ref{10.3}.

\begin{cor}
\label{11.10}
Suppose \eqref{bighyp0}, \eqref{bighyp1}, and \eqref{bighyp2} hold.
We have
\begin{multline*}
\lim_{m,k \to \infty}\lim_{X \to \infty}
   \frac{|\{\chi\in \B_{m,2k}(X) : \rk(\chi) = n\}|}{|\B_{m,2k}(X)|} \\
    = (1-\parity(E_1^+)) \Peven(n)+ \parity(E_1^+) \Podd(n),
\end{multline*}
\begin{multline*}
\lim_{m,k \to \infty}\lim_{X \to \infty}
   \frac{|\{\chi\in \B_{m,2k+1}(X) : \rk(\chi) = n\}|}{|\B_{m,2k+1}(X)|} \\
    = \parity(E_1^-) \Peven(n)+(1-\parity(E_1^-)) \Podd(n),
\end{multline*}
where the limits are over any sequence of pairs $(m,k)$ tending to infinity 
such that $\cup_X\D_{m,2k,X}$ is nonempty (for the first equality) 
and $\cup_X\D_{m,2k+1,X}$ is nonempty (for the second equality).
\end{cor}

\begin{proof}
The corollary follows directly from Theorem \ref{6.5} and Proposition \ref{10.5}.
\end{proof}

Suppose for the rest of this section that $p=2$, 
$A$ is an elliptic curve over $K$, and $T = A[2]$ with the natural twisting 
data.  Let $\Delta \in \O_{K,\Sigma}$ be the discriminant of some model of $A$; 
by \cite[Lemma 6.3]{kmr}, this $\Delta$ satisfies Lemma \ref{4.3}(i).

\begin{defn}
If $v \in \Sigma$ and $\psi,\psi' \in \Xset(K_v)$, let 
$$
h(\psi,\psi') := \dim_{\Fp}(\alpha_v(\psi)/(\alpha_v(\psi)\cap\alpha_v(\psi')))
$$
where $\alpha_v : \Xset(K_v) \to \H(q_v)$ is given by the twisting data, 
and define
$$
\gamma_v(\psi) := (-1)^{h(\one_v,\psi)}\psi(\Delta) \in \{\pm1\},
$$
$$
\delta_v = \frac{1}{|\Xset(K_v)|}\sum_{\psi\in\Xset(K_v)}\gamma_v(\psi), \quad\text{and}\quad
\delta(A/K) := \frac{(-1)^{\rk(\one)}}{2}\prod_{v \in \Sigma} \delta_v.
$$
\end{defn}

The quantity $\delta(A/K)$ is the ``disparity'' mentioned in the introduction 
(see \cite[Theorem 7.6]{kmr}).

\begin{lem}
\label{lastlem}
Suppose that $\Gal(K(A[2])/K) \cong S_3$, and that $\Sigma$ contains a prime 
$\l \nmid 2$ where $A$ has good reduction and $\Delta \notin (K_\l^\times)^2$.
Then
$$
\textstyle
\parity(E_1^+) = \frac{1}{2}-\delta(A/K) \quad\text{and}\quad
\parity(E_1^-) = \frac{1}{2}+\delta(A/K).
$$
\end{lem}

\begin{proof}
We will show that $\parity(E_1^+) + \parity(E_1^-) = 1$ and $\parity(E_1^-) - \parity(E_1^+) = 2\delta(A/K)$.

Since $|\Omega_1^+| = |\Omega_1^-| = |\Omega_1|/2$, we have
\begin{align*}
\parity(E_1^+) + \parity(E_1^-) 
   &= \frac{|\{\omega\in\Omega_1^+ : \text{$\rk(\omega)$ is odd}\}|}{|\Omega_1^+|}
   + \frac{|\{\omega\in\Omega_1^- : \text{$\rk(\omega)$ is odd}\}|}{|\Omega_1^-|} \\
   &= 2\;\frac{|\{\omega\in\Omega_1 : \text{$\rk(\omega)$ is odd}\}|}{|\Omega_1|}.
\end{align*}
Let $\l$ be as in the statement of the lemma, and fix $\varphi\in\Omega_1$ such that 
$\varphi_\l(\Delta) = -1$, and $\varphi_v = \one_v$ if $v \ne \l$.  Then multiplication 
by $\varphi$ permutes the elements of $\Omega_1$.

If $\omega \in \Omega_1$ then by 
Theorem \ref{kramer} (for the first congruence) and \cite[Lemma 5.6]{kmr} applied to the Lagrangian subspaces
$\alpha_v(\one_\l)$, $\alpha_v(\omega_\l)$, and $\alpha_v(\omega_\l\varphi_\l)$ 
(for the second congruence) we have
\begin{equation}
\label{final}
\rk(\omega\varphi) - \rk(\omega) \equiv h(\omega_\l,\omega_\l\varphi_\l) 
   \equiv h(\one_\l,\omega_\l) + h(\one_\l,\omega_\l\varphi_\l) 
   \pmod{2}.
\end{equation}
By \cite[Proposition 3]{kramer} we have 
$$
(-1)^{h(\one_\l,\omega_\l)} = \omega_\l(\Delta), \quad  
(-1)^{h(\one_\l,\omega_\l\varphi_\l)} = \omega_\l\varphi_\l(\Delta) = -\omega_\l(\Delta),
$$
so the right-hand side of \eqref{final} is odd.
Therefore $\rk(\omega)$ is odd for exactly half of the $\omega \in \Omega_1$, 
and we conclude that $\parity(E_1^+) + \parity(E_1^-) = 1$.

By Theorem \ref{kramer}, if $\omega \in\Omega_1$ we have 
$$
(-1)^{\rk(\one)+\rk(\omega)} = \prod_{v\in\Sigma}(-1)^{h(\one_v,\omega_v)} 
   = \sgn(\omega)\prod_{v\in\Sigma}\gamma_v(\omega_v).
$$
Therefore
$$
\text{$\rk(\omega)$ is odd} \iff 
\begin{cases}
\text{$\omega\in\Omega_1^+$ and $\prod_{v\in\Sigma}\gamma_v(\omega_v) \neq (-1)^{\rk(\one)}$, or} \\
\text{$\omega\in\Omega_1^-$ and $\prod_{v\in\Sigma}\gamma_v(\omega_v) = (-1)^{\rk(\one)}$.}
\end{cases}
$$
Thus
\begin{align*}
\parity(E_1^-) - &\parity(E_1^+)
   = \frac{|\{\omega\in\Omega_1^- : \text{$\rk(\omega)$ is odd}\}|}{|\Omega_1^-|}
   - \frac{|\{\omega\in\Omega_1^+ : \text{$\rk(\omega)$ is odd}\}|}{|\Omega_1^+|} \\
   &= \sum_{\omega\in\Omega_1^-}\frac{1+(-1)^{\rk(\one)}\prod_{v\in\Sigma}\gamma_v(\omega_v)}{|\Omega_1|}
      - \sum_{\omega\in\Omega_1^+}\frac{1-(-1)^{\rk(\one)}\prod_{v\in\Sigma}\gamma_v(\omega_v)}{|\Omega_1|} \\
   &= (-1)^{\rk(\one)}\frac{\sum_{\omega\in\Omega_1}\prod_{v\in\Sigma}\gamma_v(\omega_v)}{|\Omega_1|}
   = 2\delta(A/K).
\end{align*}
This proves the lemma.
\end{proof}

\begin{rem}
The assumption in Lemma \ref{lastlem} and Corollary \ref{11.11} below 
that $\Sigma$ contains a prime $\l \nmid 2$ where $A$ 
has good reduction and $\Delta \notin (K_\l^\times)^2$ can always be satisfied by adding to 
$\Sigma$ any prime in $\cP_1$.
\end{rem}

\begin{cor}
\label{11.11}
Suppose that $\Gal(K(A[2])/K) \cong S_3$, and that $\Sigma$ contains a prime $\l \nmid 2$ where $A$ 
has good reduction and $\Delta \notin (K_\l^\times)^2$.
Let $\B_m(X) := \cup_k \B_{m,k,X}$ with $\B_{m,k,X}$ as in Definition \ref{11.7}.
Then for every $n \ge 0$ we have
$$
\lim_{m \to \infty}\lim_{X \to \infty}
   \frac{|\{\chi\in \B_{m}(X) : \rk(\chi) = n\}|}{|\B_{m}(X)|} 
\textstyle
    = (\frac{1}{2}+\delta(A/K)) \Peven(n)+ (\frac{1}{2}-\delta(A/K)) \Podd(n).
$$
\end{cor}

\begin{proof}
This follows directly from Corollary \ref{11.10} and Lemma \ref{lastlem}, 
since 
$$
\phantom{\Box}\hskip 98.76pt
1-\parity(E_1^+) = \parity(E_1^-) = \frac{1}{2}+\delta(A/K).
\hskip 98.76pt\Box
$$
\renewcommand{\qedsymbol}{}
\end{proof}

\end{document}